\documentclass[a4paper,oneside,11pt]{article} % Indique que le document est de type standard ("article"), au format A4 ("a4"), en recto simple ("oneside"), avec des fontes de taille moyenne ("11pt").

\usepackage[utf8]{inputenc} % Encodage des caractères du fichier-source.
\usepackage[T1]{fontenc} % Encodage des caractères en sortie.
\usepackage{textcomp} % Jeu de symboles complémentaires.
\usepackage[english]{babel} % Support de la langue française.
\usepackage{indentfirst}
\usepackage[autolanguage]{numprint} %% Formatage des nombres.
\usepackage{hyperref} % Génère des liens hypertexte dans le fichier pdf.
\hypersetup{colorlinks=true,linkcolor=blue,citecolor=red,urlcolor=blue}
\usepackage{graphicx} % Permet l'insertion d'images.
\usepackage{verbatim} % Définit des environnements de texte préformaté et de commentaires.
\usepackage{makeidx} % Permet la création d'un index
\usepackage[refpage]{nomencl} % Permet la création d'un index des notations
\usepackage{enumitem} % Permet les listes
\usepackage{tikz} % Pour les dessins en Tikz
\usetikzlibrary{matrix,arrows,patterns}
\usepackage{multicol}
\usepackage{fullpage}
\usepackage{comment}
\usepackage{stmaryrd}
\usepackage{ulem}

\usepackage{amsmath,amssymb,amsthm,amsfonts} % Toutes ces extensions proviennent de la classe AMS-LaTeX.
\usepackage[all]{xy} % pour des diagrammes (-commutatifs)

\numberwithin{equation}{section}

\newcounter{Main}

\theoremstyle{plain} % Pour tous ce qui ressemble à des théorèmes.

\swapnumbers % Avec cette ligne, le numéro des théorèmes s'affiche avant le mot "Théorème".
\theoremstyle{definition} % Pour tout ce qui ressemble à des définitions : définitions, notations...
\newtheorem{Def}{Definition}[section] % Définitions. Avec l'option "[section]", les définitions (et tout le reste) seront numérotées en fonction de la section courante.
\newtheorem{Def,Thm}[Def]{Definition and theorem} %Pour ce qui est une définition par un théorème
\newtheorem{Def,Prop}[Def]{Proposition-definition} %Pour ce qui est une définition par un théorème
\newtheorem{Not}[Def]{Notation}
\theoremstyle{plain} % Pour tous ce qui ressemble à des théorèmes.
\newtheorem{Prop}[Def]{Proposition} % Propositions. Avec l'option "[Def]", les propositions (et tout le reste) seront numérotées collectivement avec les définitions.
\newtheorem{Lem}[Def]{Lemma} % Lemmes.
\newtheorem{Thm}[Def]{Theorem} % Théorèmes.
\newtheorem{Cor}[Def]{Corollary} % Corollaires.
\theoremstyle{remark} % Pour tout ce qui ressemble à des remarques.
\newtheorem{Rq}[Def]{Remark} % Remarques.
\renewcommand{\cal}[1]{\mathcal{#1}}
\newcommand{\bb}[1]{\mathbb{#1}}

\newcommand{\F}{\mathbb{F}}

\newcommand{\Z}{\mathbb{Z}}
\newcommand{\GL}{\mathrm{GL}}
\newcommand{\SL}{\mathrm{SL}}

\newcommand{\spec}{\mathrm{Spec}}
\newcommand{\pic}{\mathrm{Pic}}
\newcommand{\su}{\mathrm{SU}}

\newcommand{\stab}{\mathrm{Stab}}

\newcommand{\tr}{T}
\newcommand{\gal}{\mathrm{Gal}}

\newcommand{\q}{\mathfrak{q}}
\newcommand{\bk}{\mathfrak{b}}
\newcommand{\p}{\mathfrak{p}}
\newcommand{\rf}{\mathfrak{r}}
\newcommand{\cf}{\mathfrak{c}}
\newcommand{\id}{\mathrm{id}}

\title{Quotients of the Bruhat-Tits tree by arithmetic subgroups of special unitary groups} % Remplacez "Titre" par votre titre.
\author{Luis Arenas-Carmona, Claudio Bravo, Benoit Loisel, Giancarlo Lucchini Arteche} % Remplacez "Auteur" par le nom du ou des auteurs.

\date{}

\begin{document} % Début du document proprement dit.

\maketitle

\begin{abstract}
Let $K$ be the function field of a curve $C$ over a field $\bb F$ of either odd or zero characteristic. Following the work by Serre and Mason on $\SL_2$, we study the action of arithmetic subgroups of $\mathrm{SU}(3)$ on its corresponding Bruhat-Tits tree associated to a suitable completion of $K$. More precisely, we prove that the quotient graph ``looks like a spider'', in the sense that it is the union of a set of cuspidal rays (the ``legs''), parametrized by an explicit Picard group, that are attached to a connected graph (the ``body''). We use this description in order to describe these arithmetic subgroups as amalgamated products and study their homology. In the case where $\bb F$ is a finite field, we use a result by Bux, K\"ohl and Witzel in order to prove that the ``body'' is a finite graph, which allows us to get even more precise applications.\\

\end{abstract}

\section{Introduction}\label{sec intro}

Let $C$ be a smooth, projective, geometrically integral curve over a field $\F$ and let $K$ be its function field. In \cite{S}, Serre considers the action of $\SL_2(K)$ on the Bruhat-Tits tree, as well as the action of some arithmetic subgroups, such as $\SL_2(A)$, where $A$ is the ring of functions on $C$ that are regular ouside a closed point $P$. In order to study these groups, Serre gave the following description of the corresponding quotient graphs.

\begin{Thm} \cite[Ch. II, Th. 9]{S}\label{teo serre quot}
Let $X$ be the local Bruhat-Tits tree defined by the group $\text{SL}_2$ at the completion $K_P$ associated to the valuation induced by $P$. Then, the graph $\overline{X}=\mathrm{SL}_2(A) \backslash X$ is combinatorially finite, i.e. it is obtained by attaching a finite number of infinite half lines, called cusp rays, to a certain finite graph $Y$. Moreover, the set of such cusp rays is indexed by the elements of the Picard group $\mathrm{Pic}(A)=\mathrm{Pic}(\mathcal{C})/\langle P\rangle$.
\end{Thm}

Then, using Bass-Serre Theory (c.f. \cite[Ch. I, \S 5]{S}), Serre introduces a family of triples $ \lbrace (I_{\sigma}, \mathcal{P}_{\sigma}, \mathcal{B}_{\sigma}) \rbrace_{\sigma \in \mathrm{Pic}(A)}$ and a group $\cal H$ naturally associated to the action of $\SL_2(A)$ on the tree. These groups satisfy the following properties: $I_\sigma$ is an $A$-fractional ideal, $\mathcal{P}_\sigma= (\mathbb{F}^{*}\times \mathbb{F}^{*})\ltimes I_\sigma$, $\mathcal{B}_{\sigma}$ has canonical injections $\mathcal{B}_{\sigma} \rightarrow \mathcal{H}$ and $\mathcal{B}_{\sigma} \rightarrow \mathcal{P}_{\sigma}$ and $\cal H$ is finitely generated if the base field $\F$ is finite. He gets then the following result.

\begin{Thm}\cite[Ch. II, Th. 10]{S}\label{teo serre grup}
The group $\SL_2(A)$ is isomorphic to the sum of $\mathcal{P}_\sigma$, for $\sigma \in \mathrm{Pic}(A)$, and $\mathcal{H}$, amalgamated along their common subgroups $\mathcal{B}_\sigma$ according to the above injections.
\end{Thm}

In the literature, there are many particular cases of groups of the form $\SL_2(A)$ whose explicit structures has been described as amalgamated sums by describing the corresponding quotient graphs.
The most classical example was introduced by Nagao in \cite{N}. This corresponds to the case where $C=\mathbb{P}^1_{\mathbb{F}}$ and $\deg(P)=1$. In this context, up to automorphism, we can assume that $A=\mathbb{F}[t]$. By using this convention, it is not hard to see that the quotient graph is isomorphic to a ray. Consequently, it is possible to establish an isomorphism between $\SL_2(\mathbb{F}[t])$ and the amalgamated sum of the group of upper triangular matrices with coefficients in $\mathbb{F}[t]$ and $\SL_2(\mathbb{F})$ amalgamated along their intersection (see \cite{N} or \cite[Ch. II, \S 1.6]{S}).

Further developments are introduced by Serre in \cite[Ch. II, \S 2.4]{S} where he covers the cases where $C=\mathbb{P}^1_{\mathbb{F}}$ and $\text{deg}(P) \in \lbrace 2,3,4 \rbrace$, or when $C$ is a curve of genus $0$ without rational points and $\text{deg}(P)=2$. The first author of this article extends the study of Serre's and Nagao's explicit examples in \cite{Arenas}, by determining the quotient graphs when the closed point $P$ has degree $5$ or $6$, and giving a method for further computations.

The preceding results are interesting since they describe the combinatorics of $\SL_2(A)$ by characterizing the combinatorics of its corresponding quotient graph. But they can also be applied to the study its homology and cohomology. See for instance the following result by Serre.

\begin{Thm}\label{Th. Serre homology}\cite[Ch. II, \S 2.8]{S} Assume that $\mathbb{F}$ is a finite field of characteristic $p$. Let $M$ be an $\mathrm{SL}_2(A)$-module that is finitely generated as a group. Then, for each $i \geq 2$, $H_i(\mathrm{SL}_2(A), M)$ is the sum of a finite abelian group and a countable $p$-primary torsion group. Moreover, when $M$ has finite order prime to $p$, then  $H_i(\mathrm{SL}_2(A),M)$ is finite for $i\geq 0$. Also, when $M=\mathbb{Q}$, we have $H_i(\mathrm{SL}_2(A), M)=0$ for all $i \geq 2$ and $H_1(\mathrm{SL}_2(A), M)$ is a $\mathbb{Q}$-vector space of finite dimension.
\end{Thm}

One of the most significant difficulties at the moment of extending Serre's results to other reductive groups is that he extensively uses the theory of vector bundles of rank $2$ in his work, which is a particular interpretation of the Bruhat-Tits tree of $\SL_2$. Fortunately, there is a more elementary method in the literature that is easier to generalize. In \cite{M}, Mason studies the combinatorics of quotient graphs with a point of view that only requires the Riemann-Roch Theorem and some basic notions about Dedekind rings.\\

In order to present more current results, we need to introduce some technical definitions. Recall that, to every discretely valued field $K_P$ and every split reductive $K_P$-group $G$, we associate as in \cite{BruhatTits1} a polysimplicial complex $\mathcal{X}=\mathcal{X}(G,K_P)$. This complex is called the Bruhat-Tits building of $(G,K_P)$. When $G$ has rank one, for instance when $G=\SL_2$, the associated building is actually a tree. The analogs of rays in trees in higher dimensional buildings are called sectors. When $K=\bb F(t)$ and $G$ is a semisimple, simply connected, split $K$-group, Soul\'e describes in \cite{So} the topology and combinatorics of the quotient space $\overline{\mathcal{X}}:= G(\mathbb{F}[t]) \backslash \mathcal{X}$. More precisely, he shows that $\overline{\mathcal{X}}$ is isomorphic to a sector in $\mathcal{X}$, extending Nagao's results. Soul\'e consequently describes the group $G(\mathbb{F}[t])$ as an amalgamated sum of some of its subgroups. Finally, by analyzing the preceding action on $\mathcal{X}$, he obtains some results on the homology groups $H_{*}(G(\mathbb{F}[t]), F)$, for some fields $F$. His results were extended by Margaux in \cite{Margaux} to the case where $G$ is a semisimple, simply connected, isotrivial $K$-group, i.e., when $G$ splits over an extension of the form $L=\ell K$, where $\ell/\mathbb{F}$ is a finite extension.

Going to the particular case of a \emph{finite} base field $\bb F$, one of the strongest current results on the structure of quotient buildings is due to Bux, K\"ohl and Witzel (see \cite{B}). This is written in terms of a certain thin subspace of $\mathcal{X}$ that covers the quotient of the building. We state this result for further use below.

\begin{Thm}\cite[Prop 13.6]{B}\label{thm witzel}
Assume that $\mathbb{F}$ is finite. Let $G$ be an isotropic, non-commutative algebraic $K$-group and let $\mathcal{X}=\mathcal{X}(G,K)$ be the building associated to $G$ and $K$. Let $S$ be a finite set of places of $K$ and denote by $\mathcal{O}_{S}$ the ring of $S$-integers of $K$. Choose a particular realization $G\subset\mathrm{GL}_{n}$, and subsequently define $G$ as $G(K)\cap\mathrm{GL}_{n}(\mathcal{O}_{S})$. Then, there exists a constant $L$ and finitely many sectors $Q_1, \cdots, Q_s$ such that
\begin{itemize}
\item[(1)] The $G$-translates of the $L$-neighborhood of $\bigcup_{i=1}^{s} Q_i$ cover $\mathcal{X}$.
\item[(2)] For $i\neq j$, the $G$-orbits of $Q_i$ and $Q_j$ are disjoint.
\end{itemize} 
\end{Thm}

When $G$ has dimension one, then Theorem \ref{thm witzel} reduces to a special case of a theorem due to Lubotzky (\cite[Th. 6.1]{L}), who proves that the corresponding quotient graph is combinatorially finite. The proof of this result strongly involves the use of some strong theorems due to Raghunathan. Using this, Lubotzky gives a general structure theorem for lattices in $G$ and consequently confirms Serre's conjecture that such arithmetic lattices do not satisfy the congruence subgroup property.

Still over finite fields, the homological and cohomological applications of the preceding results go further than the previously described results by Serre and Soul\'e. For instance, we can cite some results by Stuhler in \cite{Stuhler}, which strongly use the combinatorial structure of the quotient of $X=\mathcal{X}(\SL_2,K_P)$ by $\Gamma= \mathrm{PGL}_2(\mathcal{O}_{S})$, where $S$ is a finite set of closed points on $C$ not including $P$. Indeed, by analyzing this quotient graph, Stuhler proves that the cohomology groups $H^{t}(\Gamma, \mathbb{F}_p)$, for $t=\mathrm{Card}(S)$, are infinite dimensional $\mathbb{F}_p$-vector spaces. This has implications on the number of generators, relations, relations among relations, and so on, for the group $\Gamma$. Finally, one of the strongest existing results on rational cohomology of split semisimple simply connected groups was proved by Harder in \cite{H2} by analyzing the actions of the latter on buildings. It essentially states that $H^v(G(\mathcal{O}_{S}), \mathbb{Q})=0$ for $v \not\in\{0, rt\}$, where $r$ is the rank of $G$ and $t=\mathrm{Card}(S)$. It also describes the dimension of the non trivial cohomology groups in terms of representations of the group $\prod_{Q \in S} G(K_Q)$. This analysis is what is currently known as reduction theory.\\

Note that almost all the previous combinatorial and homological results are specific to split groups. Thus, it is natural to seek for a generalization to quasi-split (non-isotrivial) groups. The main goal of this article is to deal with the case where $G$ is the special unitary group $\mathrm{SU}(3,h)$ associated to a three-dimensional hermitian form $h$ defined over $K$. Note that this the natural first step in this direction since $\mathrm{SU}(3,h)$ is the only quasi-split non-split simply connected semisimple group of split rank $1$ (cf. \cite[4.1.4]{BruhatTits2}). Since $\SL_2$ and $\mathrm{SU}(3,h)$ encode the behaviour of all the possible Levi subgroups of rank $1$ of quasi-split reductive groups, it is expectable that one can obtain results in the general case from results on these two cases.

Specifically, in this article we follow Mason's approach in \cite{M} in order to obtain analog results to Theorems \ref{teo serre quot}, \ref{teo serre grup} and \ref{Th. Serre homology} for $\mathrm{SU}(3,h)$. The topological implications of our results are stronger than what one can obtain from Bux, K\"ohl and Witzel's result or from Lubotzky's result, since here the ground field $\mathbb{F}$ is not assumed to be finite, but also because we get a precise control on a particular set of cusp rays in the quotient, which we call the rational cusps, and which constitute the totality of cusps in the finite field case. Moreover, by using Bass-Serre theory, we consequently describe the structure of the considered arithmetic groups as amalgamated sums. We also obtain results on the homology of these groups, for rational and finite modules. Finally, we present some interesting examples with explicit computations, getting in particular what could be considered as a ``quasi-split analog'' of Nagao's Theorem.

\section{Context and main results}\label{sec context}

Let $\F$ be a field of characteristic $\neq 2$, $C$ a smooth, projective, geometrically integral curve over $\F$ and $K$ the function field of $C$. Let $L/K$ be a quadratic extension with $\F$ algebraically closed in $L$ and denote by $\tau$ the generator of $\gal(L/K)$. This corresponds to a $2:1$ morphism of geometrically integral $\bb F$-curves $\psi:D\to C$ with $L$ the function field of $D$.

We consider the $C$-group scheme $G$ defined as follows. Consider an affine subset of $C$ given by $\spec(R)$ for some ring $R\subset K$ such that $K=\mathrm{Quot}(R)$. Let $S\subset L$ be the integral closure of $R$ in $L$. We have that $\spec(S)$ is the fiber of $\psi$ over $\spec(R)$. Let $h_R:S^3\to R$ be the $R$-hermitian form defined by
\[h_R(x_{-1},x_0,x_1):=x_{-1}\bar x_{1}+x_0\bar x_0+x_{1}\bar x_{-1},\]
where $\bar{\phantom a}$ denotes the action of $\tau\in\gal(L/K)$.
Then we define $G_R$ as the group scheme $\su(h_R)$, which is a semisimple group scheme away from ramification points (in the sense of \cite[Exp.~XIX, Def.~2.7]{SGA3III}). Since we can cover $C$ with affine subsets $\spec(R_i)$ with affine intersection, the groups $G_{R_i}$ clearly glue in order to define $G$, which we denote alternatively as $\su(h)$ (see more details on the $C$-group scheme $\su(h)$ in section~\ref{sec param}).

Let $P$ be a closed point in $C$ of degree $d$. The curve $C\smallsetminus\{P\}$ is known to be affine and thus it corresponds to $\spec(A)$ for some subring $A\subset K$ such that $K=\mathrm{Quot}(A)$. As a subset of $K$, it contains precisely the global functions that only have poles in $P$. We also consider the completion $K_P$ of $K$ with respect to the discrete valuation $\omega:=\omega_P$ induced by the point $P$. We write $\cal O_P$ for its ring of integers and $\kappa_P$ for its residue field. It is a degree $d$ extension of $\bb F$.

\emph{We assume that $L/K$ is non-split at $P$} (but it could be either ramified or unramified) and we denote by $Q$ the point in $D$ above $P$. Let $B\subset L$ be the integral closure of $A$ in $L$. Then $\spec(B)=D\smallsetminus\{Q\}$. Denote by $h_A$, $h_K$ and $h_{K_P}$ the respective induced hermitian forms on $B$, $L$ and $L_Q=L\otimes_K K_P$.

We will consider as well the reductive $K$-group $G_K=\su(h_K)$ and the $K_P$-group $G_{K_P}=\su(h_{K_P})$ obtained by base change. By our assumption on $L$ and $P$, $G_{K_P}$ and $G_K$ have semisimple rank $1$. Let $X = X(G,K_P)$ be the Bruhat-Tits tree of $G_{K_P}$ over $K_P$. The group $G_{K_P}(K_P)=G(K_P)$ acts naturally on this tree, and so do $G(K)$ and $G(A)$ as subgroups of $G(K_P)$. We are interested in the quotient graph $\overline X=X/G(A)$.\\

For every graph $\mathfrak{g}$, denote by $v(\mathfrak{g})$ and $e(\mathfrak{g})$ the corresponding sets of vertices and edges, respectively. We have then the following result on the graph $\overline X=X/G(A)$:

\begin{Thm}\label{principal result}
There exist a connected subgraph $Y$ of $\overline X$ and geodesic rays $\cf(\sigma)\subseteq \overline X$ with initial vertex $v_\sigma\in \overline X$ for each $\sigma\in\pic(B)$ such that
\begin{itemize}
\item $\displaystyle \overline X=Y\cup \bigsqcup_{\sigma\in \pic(B)} \cf(\sigma)$;
\item $v(Y)\cap v(\cf(\sigma))=\{v_\sigma\}$;
\item $e(Y)\cap e(\cf(\sigma))=\emptyset$.
\end{itemize}
Moreover, if $\bb F$ is finite, then the subgraph $Y$ is finite.
\end{Thm}

In other words, the quotient graph $\overline X$ looks like a ``spider'' with a ``body'' $Y$ and a given set of ``legs'' $\cf(\sigma)$ going to infinity (this set will be finite whenever $\pic(B)$ is finite, for instance if $\bb F$ is finite). In particular, we have that the inclusion $Y \to \bar{X}$ is a homotopy equivalence. Note however that for infinite $\bb F$ we may have an infinite body and, to our knowledge, we cannot even exclude the possibility of it having infinite diameter.\\

Recall (c.f. \cite[I.\S5.4]{S}) that to the action of $G(A)$ on $X$ one may associate a graph of groups $(G(A),\bar X)$, which consists in equipping every vertex $v$ and edge $e$ of $\bar X$ with a group (here, the stabilizer of a preimage of the corresponding edge or vertex), denoted $G_v$ and $G_e$ respectively, and homomorphisms $G_e\to G_v$ when $e$ is incident to $v$. Consider the restrictions $(G(A),Y)$ and $(G(A),\cf(\sigma))$ of $(G(A),\bar X)$ to the subgraphs $Y$ and $\cf(\sigma)$. Let $\cal H$ denote the fundamental group of $(G(A),Y)$ and $\mathcal{G}_{\sigma}$ the fundamental group of $(G(A),\cf(\sigma))$ (c.f. \cite[I.\S 5.1]{S}). Since $\cf(\sigma)$ is a ray, $\mathcal{G}_{\sigma}$ is simply the sum of the $G_v$'s for $v\in v(\cf(\sigma))$, amalgamated along the corresponding inclusions of the $G_e$'s for $e\in e(\cf(\sigma))$. Since $v_{\sigma}$ is a common vertex of $Y$ and $\cf(\sigma)$, we have injections $G_{v_{\sigma}}\hookrightarrow \cal H$ and $G_{v_{\sigma}} \hookrightarrow \mathcal{G}_{\sigma}$. So, by using Bass-Serre theory we can deduce the following result on the structure of $G(A)$ as an amalgam (see \S \ref{sec app}):

\begin{Thm}\label{main thm amalgam}
Assume that $\overline{X}$ is combinatorially finite (for instance, if $\bb F$ is finite). Then, $G(A)$ is isomorphic to the sum of $\mathcal{G}_\sigma$, for $\sigma \in \pic(B)$, and $\mathcal{H}$, amalgamated along the groups $G_{v_\sigma}$, according to the previously defined injections. Moreover, each $\mathcal{G}_{\sigma}$ is an extension of a subgroup of $\bb F^*$ by a group $H_\sigma$ which itself is an extension of two finitely generated $A$-modules.

In particular, when $\mathbb{F}$ is a finite field, we have that $\mathcal{H}$ is finitely generated and $G_{v_\sigma}$ is finite for every $\sigma\in\pic(B)$.
\end{Thm}

We can use the previous result, for example, to show that $G(A)$ is not finitely generated (cf.~Corollary \ref{cor G(A) not fg}). We also prove along the way the following result on the holomogy groups relative to $G(A)$.

\begin{Thm}\label{main thm homology}
Assume that $\mathbb{F}$ is a finite field, and denote by $p$ its characteristic. Let $M$ be a $G(A)$-module that is a finitely-generated abelian group. Then, for each $i \geq 2$, the homology group $H_i(G(A), M)$ is the sum of a finite abelian group and a countable $p$-primary torsion group. Moreover, when $M$ has finite order prime to $p$, then all groups $H_i(G(A), M)$ are finite. Also, when $M=\mathbb{Q}$, we have $H_i(G(A), M)=0$, for all $i \geq 2$, and that $H_1(G(A), M)$ is a $\mathbb{Q}$-vector space of finite dimension.
\end{Thm}

Finally, in \S\ref{examples} we illustrate these results by giving explicit computations of the quotient graph $\overline{X}$ and the structure of $G(A)$ in some cases where $C = \mathbb{P}^1_{\mathbb{F}}$ and $\deg(P)=1$. In particular, we get the following analog to Nagao's Theorem for $\SL_2$ (see Theorem \ref{thm Nagao++}):

\begin{Thm}
Let $\bb F$ be a field of characteristic $\neq 2$, let $h:\bb F[\sqrt{t}]^3\to\bb F[t]$ be the hermitian form defined by
\[h(x_{-1},x_0,x_1):=x_{-1}\bar x_1+x_0\bar x_0+x_1\bar x_{-1},\]
and let $\su(h_{\bb F[t]})$ be the subgroup of $\SL_3(\bb F[\sqrt{t}])$ of matrices that preserve $h$. Consider the following subgroups of $\su(h_{\bb F[t]})$:
\begin{align*}
F &:=\su(h_{\bb F[t]})\cap \GL_3(\bb F)\simeq \mathrm{SO}_3(\bb F),\\
\bb T &:=\{\text{upper triangular matrices in }\su(h_{\bb F[t]})\}\\
&=\left\{\begin{pmatrix}t & -\bar x & t^{-1}y\\0 & 1 & t^{-1} x \\0 & 0& t^{-1}\end{pmatrix},\ t \in \mathbb{F}^*=\bb F[\sqrt{t}]^*,\ x,y \in \mathbb{F}[\sqrt{t}],\ N(x) + T(y)=0\right\}.
\end{align*}
Then the group $\su(h_{\bb F[t]})$ is the sum of the groups $F$ and $\bb T$ amalgamated along their intersection.
\end{Thm}

There is another motivation to study the quotient graph $\overline{X}$. Classifying
hermitian vector bundles over $C$ with a fixed generic fiber is equivalent to classifying
$\su(h_K)$-orbits of $C$-lattices in a fixed hermitian space. It is often the case 
that the classification of $A$-lattices can be achieved locally, for example when
$h$ is isotropic at $P$, as in our case. Therefore, it suffices to classify $C$-lattices
with a fixed restriction to $\mathrm{Spec}(A)$. Any such lattice is completely determined by 
its $P$-component, whence their isomorphism classes are in correspondence with the vertices
of $\overline{X}$. A full description of $\overline{X}$, therefore, yields also a full 
description of the set of such bundles. This is a hard problem in general. See \S\ref{section lambda 0 infty} for some computations of the valency at a special vertex.

\begin{Rq}
Looking back on Theorem \ref{thm witzel}, or the results by Lubotzky and Harder mentioned above, one could wonder whether our results could be extended to rings of the form $\cal O_S$, where $S$ is a finite set of closed points in $C$. We leave this study for later mainly because of the following issue. Since we are restricting our study to trees by following Mason's approach, we are avoiding the general theory of buildings in order to gain in simplicity. However, for any closed point $P$ in $C$ that is split in $D$ we have $G_{K_P}=\SL_{3,K_P}$, so that it is not a tree but a 2-dimensional building that is associated to this point (this is why we assume $P$ to be non-split above). We could still try to study the case of several non-split points, but in that case an application of the Strong Approximation Theorem tells us that the quotient graph of the tree associated to $G(K_P)$ for $P\in S$ is simply an edge. The natural thing to do then is actually studying the action of $G(\cal O_S)$ on the product of all such trees, but then we are back to higher-dimensional buildings.

Note however that the case of a single point is the crucial one for applications in the study of such groups as amalgams. Indeed, we can write a group of the form $G(\cal O_S)$ as an amalgam of two groups of the form $G\left(\cal O_{S-\{P\}}\right)$
by considering the local quotient at $P$, which, as mentioned above, 
is an edge. When $S$ has two points, one of such subgroups is the one being considered here, while the other is the stabilizer of a lattice of prime discriminant. See \S\ref{sec lattices} for the relation between trees and lattices.
\end{Rq}

\section{Notations and preliminaries}
\subsection{Conventions on graphs and trees}\label{sec trees}

In this section we recall some basic notions about graphs and trees. A graph $\mathfrak{g}$ consists of a pair of sets $v=v(\mathfrak{g})$ and $e=e(\mathfrak{g})$, and three functions $s,t:e\rightarrow v$ and $r:e\rightarrow e$ satisfying the identities $r(a)\neq a$ , $r\big(r(a)\big)=a$ and $s\big(r(a)\big)=t(a)$, for every $a \in e$. The set $v$ and $e$ are called set of vertices and edges respectively, and the functions $s,t$ and $r$ are called respectively source, target and reverse. A simplicial map $\gamma:\mathfrak{g} \rightarrow \mathfrak{g}'$ between graphs is a pair of functions
$\gamma_v:v(\mathfrak{g})\rightarrow v(\mathfrak{g}')$ and $\gamma_e:e(\mathfrak{g})\rightarrow e(\mathfrak{g}')$  preserving these functions, and a similar convention applies to group actions. An action of a group $\Gamma$
on a graph $\mathfrak{g}$ does not have inversions if $g.a\neq r(a)$ for every edge $a$ and every element $g\in\Gamma$. An action without inversions defines a quotient graph in the usual sense.

Let $\mathfrak{g}$ be a graph. A \textbf{ray} $\mathfrak{r}$ in $\mathfrak{g}$ is a subcomplex of $\mathfrak{g}$ whose sets of vertices $\lbrace v_i \in v(\mathfrak{r}): i\in \mathbb{Z}_{\geq 0}\rbrace$ and edges $\lbrace e_i \in e(\mathfrak{r}): i\in \mathbb{Z}_{\geq 0}\rbrace$ satisfy the identities $s(e_i)=v_i$, $t(e_i)=v_{i+1}$ and $v_i \neq v_{j}$ for all $i \neq j$ in $\mathbb{Z}_{\geq 0}$. Let $\mathfrak{r}_1$ and $\mathfrak{r}_2$ be rays whose set of vertices are $\left\lbrace v_i^{(1)}\in v(\mathfrak{r}_1): i\in \mathbb{Z}_{\geq 0}\right\rbrace$ and $\left\lbrace v_i^{(2)} \in v(\mathfrak{r}_2): i\in \mathbb{Z}_{\geq 0}\right\rbrace$ respectively. We say that $\mathfrak{r}_1$ and $\mathfrak{r}_2$ are equivalent, and we write $\mathfrak{r}_1 \sim \mathfrak{r}_2$, if there exists $t, i_0 \in \mathbb{Z}_{\geq 0}$ such that $v_{i}^{(1)}= v_{i+t}^{(2)}$, for all $i \geq i_0$. This definition applies to $\mathfrak{g}=X$, $\mathfrak{g}=\bar{X}$ or any of their subcomplexes. The equivalence class of a ray $\rf$ is denoted $\partial_{\infty}(\rf)$, and it is called the \textbf{visual limit} of $\rf$. Analogously, the set of equivalence classes of rays is denoted by $\partial_{\infty}(X)$. The same definitions and notation applies to any unbounded subtree of $X$.

A \textbf{cusp ray} in a graph $\mathfrak{g}$ is a ray where every non-initial vertex has valency two in $\mathfrak{g}$. We say that a graph is \textbf{combinatorially finite} if it is the union of a finite graph and a finite number of cusp rays. By definition a cusp in $\mathfrak{g}$ is an equivalence class of cusp rays. In the context of Theorem \ref{principal result} a set of representatives for the (rational) cusps is precisely $\lbrace \cf(\sigma): \sigma \in \pic(B) \rbrace$.

\subsection{Parametrization of subgroups of \texorpdfstring{$\mathrm{SU}(h)$}{SU(h)}}\label{sec param}

We recall some known facts on the algebraic $K$-group $G_K = \mathrm{SU}(h_K)$ (which are actually true over fields of arbitrary characteristic, see for instance \cite[4.1]{BruhatTits2}, or \cite[§4 Case 2. p. 43--50]{Landvogt}) and we extend some of these to the $C$-group scheme $G$ while fixing notations on the way.\\

We will use repeatedly the norm and trace map of the quadratic extension $L/K$, denoted by $N = N_{L/K}$ and $T = \operatorname{Tr}_{L/K}$, respectively. It is also convenient to introduce the following $K$-varieties:
\begin{align*}
H(L,K) &:= \left\{ (u,v) \in R_{L/K}(\mathbb{G}_{a,L})^2,\ N(u) + T(v) = 0\right\},\\
H(L,K)^0 &:= \left\{ (0,v) \in 0 \times R_{L/K}(\mathbb{G}_{a,L}),\ T(v) = 0\right\},
\end{align*}
We will abusively denote by $H(L,K)$ and $H(L,K)^0$ the corresponding sets of $K$-points, which we will naturally view as subsets of $L^2$. Note that, if $(u,v) \in H(L,K)$, then we also have $(\bar{u},v),(u,\bar{v}),(\bar{u},\bar{v}) \in H(L,K)$. Moreover, one can endow these varieties with a natural group structure from a matrix realization given further below.\\

We recall that $G_K$ is a quasi-split semisimple $K$-group which admits a natural $K$-linear faithful representation $G_K \to R_{L/K}(\mathrm{SL}_{3,L})$. Since $D\to C$ is finite and flat, the Weil restriction $R_{D/C}(\mathrm{SL}_{3,C})$ is well-defined and it is easy to see that the representation above extends to an embedding $G\to R_{D/C}(\mathrm{SL}_{3,C})$ (actually, $G$ is the schematic closure of $G_K$ in $R_{D/C}(\mathrm{SL}_{3,C})$). Through this representation, we identify $G$ with its image and the elements of $G(K)$ (resp.~$G(A)$, $G(K_P)$, $G(\cal O_P)$) with $3\times 3$ matrices with coefficients in $L$ (resp.~$B$, $L_Q$, $\cal O_Q$) preserving the hermitian form and with trivial determinant.

Given this matrix interpretation of $K$-points and $A$-points of $G$, we see that the group $G(C)\subset G(K)$ of $C$-points corresponds to matrices in $\SL_{3,L}$ with coefficients in $\bb F$. In particular, we will be using the fact that $G(C)=G(A) \cap G(\mathcal{O}_P)$. The same goes for $\mathcal{T}(C)$, which corresponds to diagonal matrices in $G(C)$.

A maximal $K$-split torus $\mathcal{S}$ of $G$ is given by the matrices of the form:
\[
    (2a)^\vee(s) = \begin{pmatrix} s & 0& 0\\ 0 & 1 & 0 \\ 0 & 0 & s^{-1} \end{pmatrix} \qquad s \in \mathbb{G}_{m,K}
\]
and its centralizer $\mathcal{T} = \mathcal{Z}_G(\mathcal{S})$, that is a maximal $K$-torus since $G$ is quasi-split, admits the following parametrization with diagonal matrices:
\[
    \begin{array}{cccc}
    \widetilde{a}:& R_{L/K}(\mathbb{G}_{m,L}) &\to  & \mathcal{T} \subset G\\
     & t &\mapsto & \begin{pmatrix} t & 0& 0\\ 0 & \bar{t} t^{-1} & 0 \\ 0 & 0 & \bar{t}^{-1} \end{pmatrix}
    \end{array}.
\]
Note that $\widetilde{a} : R_{L/K}(\mathbb{G}_{m,L}) \to \mathcal{T}$ is the extension of $(2a)^\vee : \mathbb{G}_{m,K} \to \mathcal{S} \subset \mathcal{T}$.

The maximal $K$-torus $\mathcal{T}$, and therefore $G$, splits over $L$ onto the torus:
\[
	\mathcal{T}_L = \left\{
		\begin{pmatrix}
			x & & \\ & y & \\ & & z
		\end{pmatrix},\
		xyz = 1 \right\}
		\cong \mathbb{G}_{m,L}^2
\]
having a basis of characters over $L$ given by:
\[
	\alpha
		\begin{pmatrix}
			x & & \\ & y & \\ & & z
		\end{pmatrix}
		= y z^{-1}
	\qquad \text{ and } \qquad
	\bar{\alpha}
		\begin{pmatrix}
			x & & \\ & y & \\ & & z
		\end{pmatrix}
		= x y^{-1}.
\]

We denote the restriction of $\alpha$ to $\mathcal{S}$,
which coincides with that of $\bar{\alpha}$ by $a = \alpha|_{\mathcal{S}} = \bar{\alpha}|_{\mathcal{S}}$;
it is the character generating $X^*_K(\mathcal{S})$.
We denote the restriction of $\alpha + \bar{\alpha}$ to $\mathcal{S}$ by $2a = (\alpha + \bar{\alpha})|_{\mathcal{S}}$.

The $K$-root system of $G$ is of type $BC_1$ and we write it as $\Phi = \{\pm a, \pm 2a\}$, so that $(2a)^\vee$ is the coroot of $2a$, and it generates the $\mathbb{Z}$-module of cocharacters of $\mathcal{S}$.
We pick the Borel $K$-subgroup $\mathcal{B}$ of $G$ consisting in upper-triangular matrices, which contains $\mathcal{T}$.
It corresponds to positive roots $\{a, 2a\}$.
We parametrize root groups by:
\[
    \begin{array}{cccc}
    \mathrm{u}_a:& H(L,K) &\to  & \mathcal{U}_a \subset G\\
     & (u,v) &\mapsto & \begin{pmatrix} 1 & -\bar{u} & v \\ 0 & 1 & u \\ 0 & 0 & 1\end{pmatrix}
    \end{array}
    \qquad \text{ and } \qquad
    \begin{array}{cccc}
    \mathrm{u}_{2a}:& H(L,K)^0 &\to  & \mathcal{U}_{2a} \subset G\\
     & v &\mapsto & \begin{pmatrix} 1 & 0 & v \\ 0 & 1 & 0 \\ 0 & 0 & 1\end{pmatrix}
    \end{array}.
\]
The Weyl group $W(G,\cal S) = \mathcal{N}_G(\cal S) / \mathcal{Z}_G(\cal S)$ is of order $2$.
We denote by $\mathrm{s}$ the lift of the non-trivial element, that exchanges $a$ with $-a$, and whose matrix realization is
\[
    \mathrm{s} = \begin{pmatrix} 0 & 0 & -1\\ 0 & -1 & 0\\ -1 & 0& 0\end{pmatrix}.
\]
The parametrization of root groups for a negative root is given by the formula
\[
    \mathrm{u}_{-a}(u,v) = \mathrm{s} \mathrm{u}_a(u,v) \mathrm{s}.
\]
Matricially, this gives:
\[
    \begin{array}{cccc}
    \mathrm{u}_{-a}:& H(L,K) &\to  & \mathcal{U}_{-a} \subset G\\
     & (u,v) &\mapsto & \begin{pmatrix} 1 & 0 & 0 \\ u & 1 & 0 \\ v & -\bar{u} & 1\end{pmatrix}
    \end{array}
    \qquad \text{ and } \qquad
    \begin{array}{cccc}
    \mathrm{u}_{-2a}:& H(L,K)^0 &\to  & \mathcal{U}_{-2a} \subset G\\
     & v &\mapsto & \begin{pmatrix} 1 & 0 & 0 \\ 0 & 1 & 0 \\ v & 0 & 1\end{pmatrix}
    \end{array}.
\]

\section{Stabilizers of vertices}\label{sec stab}

\begin{Def}\label{def g}
Define, for $(u,v)\in H(L,K)$, the element:
\[g_{u,v} = \mathrm{u}_{-a}(u,v) \mathrm{s} = \mathrm{s} \mathrm{u}_{a}(u,v)=
    \begin{pmatrix}
        0 & 0& -1\\
        0 & -1 & -u\\
        -1 & \bar u & -v
    \end{pmatrix}\in G(K)
\]
This element will be used repeatedly in this article.
\end{Def}

Using these elements, we can parametrize some geodesic rays in $X$ as follows.
Recall that a standard apartment is defined from $\mathcal{S}$ by $\mathbb{A} = X_*(\mathcal{S}) \otimes \mathbb{R} \cong \mathbb{R} a^\vee$ with vertices $x = r a^\vee$ satisfying $a(x) = 2r \in \Gamma_a = \frac{1}{2} \omega(L_P^\times)$ \cite[4.2.21(4) and 4.2.22]{BruhatTits2}.
If $e_P$ denotes the ramification index of $L_P/K_P$ and if the valuation $\omega$ in $P$ is normalized such that $\omega(K_P^*) = \mathbb{Z}$, then $r \in \frac{1}{4e_P} \mathbb{Z}$.

We recall from \cite[6.2.10]{BruhatTits1} and \cite[4.2.7]{BruhatTits2} that the action of $\mathcal{N}_G(\mathcal{S})(K_P)$ onto $\mathbb{A}$ can be deduced from the formulas:
\[
    \widetilde{a}(v) \cdot x = x - \frac{1}{2} \omega(v) a^\vee
    \qquad \text{ and } \qquad
    \mathrm{s} \cdot x = -x
\]
Indeed, the formula $\langle\nu(t),a\rangle = - \omega \circ a (t)$ from \cite[4.2.7]{BruhatTits2}, applied with $t = \widetilde{a}(v)\in \mathcal{T}(K)$ and $\nu(t) = r a^\vee$, gives that $2r = - \omega\Big(a\big(\widetilde{a}(v)\big)\Big) = - \omega\Big(v^2 (\bar{v})^{-1}\Big)$.

\begin{Def}
We define the half-apartment $\rf(\infty)$ of $\mathbb{A}$ by
\[\rf(\infty):=\{x \in \mathbb{A},\ a(x) \geqslant 0\}.\]
We define moreover a numbering of its vertices $v(\rf(\infty))$ as follows:
\[\lambda_n(\infty):=\frac{n}{4e_P} a^\vee \in \rf(\infty),\quad \text{for } n \in \mathbb{Z}_{\geqslant 0}.\]

For any $(u,v) \in H(L,K)$, we define another half-apartment $\rf(u,v) = g_{u,v}^{-1} \cdot \rf(\infty)$ with a numbering of its vertices by $\lambda_n(u,v) = g_{u,v}^{-1} \cdot \lambda_n(\infty)$.
\end{Def}

\begin{Lem}\label{lem density of H}
The set $H(L,K)$ is a dense subset of $H(L_Q,K_P)$.
\end{Lem}

\begin{proof}
$H(L,K)$ is a $K$-split unipotent group of dimension $3$, whence isomorphic to $\mathbb{A}^3_K$ as $K$-variety.
\end{proof}

\begin{Lem}\label{lem cover}
\[X = \bigcup_{(u,v) \in H(L,K)} \{1,\mathrm{s}\}\cdot \rf(u,v)\]
\end{Lem}

\begin{proof}
Throughout this proof, we denote $\rf(0,0)$ by $\rf(-\infty)$. Then $\rf(\infty)= \mathrm{s} \cdot \rf(-\infty)\subset \mathbb{A}$.

Let $x \in X$ be any point. If $x \in \rf(\pm\infty)$, we are done. Otherwise, since $\rf(\infty)$ and $\rf(-\infty)$ are the sectors of $\mathbb{A}$ with tip $\lambda_0(\infty) = 0 a^\vee$, we know that there is an apartment $\mathbb{A}_x$ that contains $x$ and $\rf(\varepsilon \infty)$ for $\varepsilon \in \{\pm\}$. If $\bb{A}_x$ contains $\rf(\infty)$, then we know that $\mathrm{s} \cdot\bb{A}_x$ contains both $\rf(-\infty)$ and $\mathrm{s} \cdot x$.
Thus, we only need to prove that if $\bb A_x$ contains $\rf(-\infty)$, then $x\in \rf(u,v)$ for certain $(u,v)\in H(L,K)$.

We know by \cite[9.7(i)]{Landvogt} that the valued root group $U_{-a,\rf(- \infty)}$ acts transitively on the set of apartments of $X$ containing $\rf(-\infty)$.
Thus, there exists an element $(\tilde{u},\tilde{v}) \in H(L_P,K_P)$ such that $x' = \mathrm{u}_{- a}(\tilde{u},\tilde{v}) \cdot x \in \rf(\infty)$.
We know that the stabilizer of $x'$ in $U_{- a, \rf(- \infty)}$ is the open subgroup $U_{-a,x'}$ of $U_{-a,\rf(-\infty)}$ and that $U_{-a,\rf(-\infty)}$ acts continuously.
By continuity of the parametrization and by density of $H(L,K)$ in $H(L_P,K_P)$ (Lemma \ref{lem density of H}), there exists an element $(u,v) \in H(L,K)$ such that $\mathrm{u}_{- a}(u,v) \cdot x = x' \in \rf( \infty)$.
Thus $x \in \mathrm{u}_{-a}(u,v)^{-1} \cdot \rf(\infty) = \mathrm{s} (g_{u,v}^{-1}) \cdot \rf(\infty) = \mathrm{s} \cdot \rf(u,v)$.
\end{proof}

In the next lemmas, we observe that it is convenient to introduce the following fractional ideals, subrings and subsets:

\begin{Not}
For any $(u,v) \in H(L,K)$, we set $\mathfrak{p}_{u,v}$ to be the fractional $B$-ideal $B + Bu + Bv\subset L$ and $\mathfrak{b}_{u}:= B+Bu$,
so that:
\[
    \mathfrak{p}_{u,v}^{-1} =
    \begin{cases}
        B & \text{ if } v=0,\\
        B \cap \frac{1}{v} B & \text{ if } u= 0,\ v \neq 0,\\
        B \cap \frac{1}{v} B \cap \frac{1}{u} B & \text{ if } u\neq 0.
    \end{cases}
    \qquad \text{ and } \qquad
    \mathfrak{b}_u^{-1} = 
    \begin{cases}
        B & \text{ if } u=0,\\
        B \cap \frac{1}{u} B & \text{ if } u\neq 0.
    \end{cases}
\]
We define as well the fractional $B$-ideals
\[
    {\mathfrak{q}_{u,v}}^{-1}
    = \mathfrak{p}_{\bar{u},v} \mathfrak{p}_{u,\bar{v}}
    = B+Bu+Bv+B\bar{u}+B\bar{v}+Bu\bar{u}+Buv+B\bar{u}\bar{v}+Bv\bar{v}
\]
and
\[
    \widetilde{\mathfrak{q}_{u,v}}^{-1}
    = \mathfrak{b}_u \mathfrak{p}_{u,\bar{v}}
    = B+B\bar{v}+Bu+Bu^2+Bu\bar{v}
\]
so that
\[
    \mathfrak{q}_{u,v}
    = \begin{cases}
        B & \text{ if } v = 0,\\
        B \cap B\frac{1}{v}\cap B\frac{1}{\bar{v}}\cap B\frac{1}{v\bar{v}} & \text{ if } v \neq 0 \text{ and } u = 0,\\
        B\cap B\frac{1}{u} \cap B\frac{1}{v}\cap B\frac{1}{\bar{u}}\cap B\frac{1}{\bar{v}}\cap B\frac{1}{u\bar{u}}\cap B\frac{1}{uv}\cap B\frac{1}{\bar{u}\bar{v}}\cap B\frac{1}{v\bar{v}} & \text{ if } u \neq 0.\\
    \end{cases}
\]
and
\[
    \widetilde{\mathfrak{q}_{u,v}}
    = \begin{cases}
        B & \text{ if } v = 0,\\
        B \cap B\frac{1}{\bar{v}} & \text{ if } v \neq 0 \text{ and } u = 0,\\
        B\cap B\frac{1}{u} \cap B\frac{1}{\bar{v}}\cap B\frac{1}{u^2}\cap B\frac{1}{u\bar{v}} & \text{ if } u \neq 0.\\
    \end{cases}
\]
\end{Not}

\begin{Def}
We define the degree of a fractional $B$-ideal $\q$ as
\[\deg\q:=\sum_{P'\neq P}d_{P'}v_{P'}(\q),\]
where $d_{P'}$ denotes the degree of the closed point $P'\in C$ and $v_{P'}$ denotes the corresponding valuation. Note that $v_{P'}(\q)$ makes sense since the completion $\q_{P'}\subset K_{P'}$ is a principal fractional ideal generated by a power of a uniformizer.
\end{Def}

The obvious properties of $\deg$ that we will use in the results below are:
\begin{itemize}
\item $\deg \q\q'=\deg \q+\deg\q'$.
\item $\deg \q\leq \deg \q'$ if $\q\supset\q'$.
\item $\deg \q=\deg\bar\q$ for any $\q$.
\end{itemize}

\begin{Lem}\label{lem Idealic}
Let $(u,v) \in H(L,K)$.
If $n > \frac{2e_P}{d} \deg \widetilde{\mathfrak{q}_{u,v}}$, then
\[
    {\mathfrak{q}_{u,v}}^{-1} \cap \omega^{-1}\Big(\big[\frac{n}{e_P},\infty\big]\Big) = \{0\}
    \qquad \text{ and } \qquad
    \widetilde{\mathfrak{q}_{u,v}}^{-1} \cap \omega^{-1}\Big(\big[\frac{n}{2e_P},\infty\big]\Big) = \{0\}
.\]
\end{Lem}

\begin{proof}
Write $\q$ for either $\q_{u,v}$ or $\widetilde{\q_{u,v}}$ and let $x\in \q^{-1}$ be a non-zero element. Then the product formula gives
\[d\omega(x)+\sum_{P'\neq P}d_{P'}v_{P'}(x)=0.\]
In particular, if we denote by $(x)=Bx$ the fractional ideal generated by $x$, we get $d\omega(x)=-\deg(x)$ and, since clearly $\deg \q^{-1}\leq\deg(x)$, we get $d\omega(x)\leq-\deg\q^{-1}=\deg\q$. When $\q=\widetilde{\q_{u,v}}$, we see from the hypothesis on $n$ that
\[\omega(x)\leq \frac{1}{d}\deg\q_{u,v}<\frac{n}{2e_P}.\]
This implies the second identity. In order to get the first one, it will suffice to prove that $\deg\q_{u,v}\leq 2\deg\widetilde{\q_{u,v}}$. Now this falls easily from the definitions and the basic properties of $\deg$. Indeed, we have
\[\deg\widetilde{\q_{u,v}} =-\deg \mathfrak{b}_u -\deg\mathfrak{p}_{u,\bar v}=\deg \bk^{-1}_u-\deg\p_{u,\bar v}\geq -\deg\p_{u,\bar v},\]
the last inequality coming from the fact that $\bk_u^{-1}\subset B$, hence it has positive degree. And since $\p_{\bar u,v}=\overline{\p_{u,\bar v}}$, we get
\[\deg\widetilde{\q_{u,v}}\geq -\deg\p_{u,\bar v}=-\frac 12(\deg \p_{u,\bar v}+\deg \p_{\bar u,v})=\frac 12\deg \q_{u,v}.\]
This concludes the proof.
\end{proof}

%\begin{Rq}\label{rem q and q tilde}
%Note that, when $u=0$, we have $\mathfrak{b}_u^{-1}=B$ and hence the proof above gives us the equality $\deg\q_{u,v}=2\deg\widetilde{\q_{u,v}}$.\\
%\end{Rq}

Having Lemma \ref{lem cover} at hand, and recalling that $\mathrm{s}\in G(A)$. we can study stabilizers of vertices in $X$ for the action of $G(A)$ by analyzing the stabilizers of vertices of the form $\lambda_n(u,v)$ with $(u,v)\in H(L,K)$. The results that follow, and in particular Proposition \ref{prop stab}, can be seen then as an analogue of \cite[Lemma 3.4]{M} in the context where $G=\su(h)$.

We start however by studying a simpler type of subgroup, which are the stabilizers of the visual limits of the rays $\rf(u,v)$. In order to proceed, we need to fix some notations.\\

Recall that $G(C)$ and $\mathcal{T}(C)$ denote the subgroups of $G(K)$ and $\mathcal{T}(K)$ whose coefficients in the matrix representation are contained in $\bb F\subset K$. In particular, $\mathcal{T}(C)$ consists of elements $\widetilde{a}(t)$ with $t \in \mathbb{F}^\times \subset K^\times$. Let $\mathbb{B}$ be the subgroup of $G(K)$ that is the internal semi-direct product $\mathbb{B} := \mathcal{U}_a(K) \rtimes \mathcal{T}(C)$, i.e.~elements of the form $\mathrm{u}_{a}(x,y) \widetilde{a}(t)$ with $(x,y)\in H(L,K)$ and $t\in\bb F^\times \subset L^\times$.

\begin{Def}[Stabilizers of visual limits]
For $(u,v) \in H(L,K)$, define the group
\[\operatorname{Stab}(u,v)= g_{u,v}^{-1} \mathbb{B} g_{u,v} \cap G(A).\]
Denote by $U(u,v)$ the subgroup of its unipotent elements (it is a group because $\bb B$ is contained in $\cal B(K)$ and $\cal B$ is solvable).
%and by $S(u,v)$ the subset of its semisimple elements.
\end{Def}

The fact that this subgroup is indeed the stabilizer of the visual limit of $\rf(u,v)$ will be proved in Proposition \ref{prop stab}. We start by studying in detail the subsets $U(u,v)$. In order to do so, note that $H(L,K)^0$ can be seen as an $A$-module via $b\cdot (0,v):=(0,bv)$. On the other hand, the projection onto the first coordinate $\pi_1 : H(L,K) \to L$ is a surjective group homomorphism with kernel $H(L,K)^0$. The quotient group $H(L,K)/H(L,K)^0$ is also naturally endowed with an $A$-module structure by $b \cdot (x, -) = (b x, -)$ via the induced isomorphism.

\begin{Lem}\label{Lemma M}
Let $(u,v) \in H(L,K)$.
\begin{enumerate}
    \item There is a unique proper and finitely generated $A$-submodule $H(u,v)^0$ of $H(L,K)^0\cong K$ such that 
\[ g_{u,v}^{-1} \mathcal{U}_{2a}(K) g_{u,v} \cap G(A) = \left\{ g_{u,v}^{-1} \mathrm{u}_a\big(x,y\big) g_{u,v},\ (x,y) \in H(u,v)^0\right\}.\]
More precisely, we have $H(u,v)^0 = H(L,K)^0 \cap \left(0 \times  \mathfrak{q}_{u,v}\right)$, which is a nontrivial $A$-module.
    \item There is a unique subgroup $H(u,v)$ of $H(L,K)$ such that 
\[ U(u,v) = g_{u,v}^{-1} \mathcal{U}_{a}(K) g_{u,v} \cap G(A) = \{ g_{u,v}^{-1} \mathrm{u}_a\big(x,y\big) g_{u,v},\ (x,y) \in H(u,v)\}.\]
Moreover, $H(u,v)$ contains $H(u,v)^0$ as a normal subgroup, is contained in $L \times B$, and
the quotient group $H(u,v) / H(u,v)^0$ canonically identifies, as projection onto the first coordinate, with a finitely generated (and therefore proper) $A$-submodule of $H(L,K)/H(L,K)^0\cong L$.

\item The subset $U(u,v)$ is a normal subgroup of $\stab(u,v)$ and the quotient group $\stab(u,v) / U(u,v)$ is isomorphic to a subgroup of $\mathbb{F}^\times$.
\end{enumerate}

\end{Lem}

\begin{proof}
For any subset $X$ of $H(L,K)$, denote by $\mathrm{u}_{a}(X) = \lbrace \mathrm{u}_a(x,y): (x,y) \in X\rbrace$.
Note that the subgroup $\mathcal{U}_{2a}(K)$ of $\mathcal{U}_{a}(K)$ is equal to $\mathrm{u}_a\big( H(L,K)^0 \big)$.

The existence of subgroups $H(u,v)^0 \subset H(L,K)^0$ and $H(u,v) \subset H(L,K)$ is immediate since $\mathrm{u}_a$ realizes group isomorphisms $H(L,K) \to \mathcal{U}_{a}(K)$ and $H(L,K)^0 \to \mathcal{U}_{2a}(K)$.
Moreover, $H(u,v)^0$ is a normal subgroup of $H(u,v)$ as inverse image by the group isomorphism $\mathrm{u}_{a}$ of the normal subgroup $\mathcal{U}_{2a}(K) \cap g_{u,v} G(A) g_{u,v}^{-1}$ of $\mathcal{U}_{a}(K) \cap g_{u,v} G(A) g_{u,v}^{-1}$.

The equality $ U(u,v) = g_{u,v}^{-1} \mathcal{U}_{a}(K) g_{u,v} \cap G(A)$ is immediate by definition since $\mathcal{U}_a(K)$ are the unipotent elements of $\mathbb{B}$.
Moreover, since $\mathcal{U}_a(K)$ is a normal subgroup of $\mathbb{B}$, we deduce that $U(u,v)$ is a normal subgroup of $\stab(u,v)$.

Note that $g_{u,v}^{-1}\mathrm{u}_{a}(x,y) g_{u,v} \in U(u,v)$ if and only if $(x,y) \in H(L,K)$ and $\mathfrak{n}\in \mathcal{M}_3(B)$, where
\begin{equation}\label{eqn matrix n}
    \mathfrak{n}=\operatorname{Mat}\Big(g_{u,v}^{-1}\mathrm{u}_{a}(x,y)g_{u,v}\Big)- \text{Id} 
    = \textnormal{\footnotesize$\begin{pmatrix}
        \bar{v}y+\bar{u}x    &  -\bar{v}\bar{x}-\bar{u}\bar{v}y -N(u)x  &   -u\bar{v}\bar{x}+N(v)y+\bar{u}vx  \\
        -uy+x   &   u\bar{x}+N(u)y-\bar{u}x   &  u^2\bar{x}-vuy+vx    \\
        y   & -\bar{x}-y\bar{u} &   u\bar{x}+yv \\
    \end{pmatrix}$\normalsize}.
\end{equation}
Consider the matrix
\[\mathfrak{m} = \operatorname{Mat}(g_{u,v}^{-1}) \begin{pmatrix} 0 & 0& 1\\0&0&0\\0&0&0\end{pmatrix} \operatorname{Mat}(g_{u,v}) = \begin{pmatrix}
\bar v & -\bar u \bar v & N(v)\\
-u & N(u) & -uv\\
1 & -\bar u & v
\end{pmatrix}.\]
Note that $\mathfrak{n}^2 = - N(x) \mathfrak{m}.$

If $x = 0$, then $\mathfrak{n} = y \mathfrak{m}$.
Thus $\mathfrak{n} \in \mathcal{M}_3(B)$ if, and only if, $y\in \mathfrak{q}_{u,v}$.
This proves that $H(u,v)^0 = H(L,K) \cap \left( 0 \times \mathfrak{q}_{u,v}\right)$. In order to prove that $H(u,v)^{0}$ is non trivial, we only need to check that there exists a nonzero element $z \in \mathfrak{q}_{u,v}$ such that $T(z)=0$. Now, note that $\overline{\mathfrak{q}_{u,v}} = \mathfrak{q}_{\bar{u}, \bar{v}}= \mathfrak{q}_{u,v}$. Thus, the trace of any element in $\mathfrak{q}_{u,v}$ belongs to $\mathfrak{q}_{u,v}$. Note that $\mathfrak{q}_{u,v}$ is not contained in $K$, since it is a $B$-ideal. So, let $z_0 \in \mathfrak{q}_{u,v} \smallsetminus K$. Then $z= z_0-\frac{1}{2} \tr(z_0) \in \mathfrak{q}_{u,v}$ is a nontrivial element with trivial trace.\footnote{If we assumed $\operatorname{char}(\mathbb{F}) = 2$, since $\mathfrak{q}_{u,v}$ is a $B$-ideal stable by conjugation, we could take any $z_0 \in \mathfrak{q}_{u,v} \setminus \{0\}$ and consider $z= N(z_0) \in \mathfrak{q}_{u,v} \cap K \setminus \{0\}$ (instead of $z_0 + \frac{1}{2} T(z_0)$) which is of trace $T(z)=0$.} Recalling that, since $A$ is a Noetherian domain, an $A$-submodule of a finitely generated $A$-module is finitely generated, this proves statement~1.\\

Consider any element $(x,y) \in H(u,v) \setminus H(u,v)^0$ so that $-T(y) = N(x) \neq 0$.
Let $t \in A$ be any element.
We want to prove the existence of an element $y_t\in L$ such that $(tx,y_t) \in H(u,v)$.
Write $y_t = ty + z_t N(x)$ with $z_t \in L$.
Then
\begin{equation*}
    N(tx) + T(y_t) = t^2 N(x) +t T(y) + N(x) T(z_t) = N(x) \big( t^2 - t + T(z_t) \big)
\end{equation*}
Moreover
\begin{equation*}
    \operatorname{Mat}\big(g_{u,v}^{-1} \mathrm{u}_a(tx,y_t) g_{u,v}\big) = \operatorname{Id} + t \mathfrak{n} +z_t \mathfrak{n}^2
\end{equation*}
Thus $\operatorname{Mat}\big( g_{u,v}^{-1} \mathrm{u}_a(tx,y_t) g_{u,v} \big) \in \mathcal{M}_3(B) \Longleftrightarrow z_t \mathfrak{n}^2 = -z_t N(x) \mathfrak{m} \in \mathcal{M}_3(B)$
since $t \in A \subset B$ and $\mathfrak{n}\in \mathcal{M}_3(B)$ by assumption on $(x,y)$.
Hence 
\begin{equation}\label{equivalence H(u,v)}
    (tx,y_t) \in H(u,v) \Longleftrightarrow
    T(z_t) = t - t^2
    \text{ and }
    z_t \in \frac{1}{N(x)} \mathfrak{q}_{u,v}.
\end{equation}
Note that $B \subset \frac{1}{N(x)} \mathfrak{q}_{u,v}$ since $\mathfrak{n}^2 = - N(x) \mathfrak{m} \in \mathcal{M}_3(B)$ and $\mathfrak{q}_{u,v}$ is a fractional $B$-ideal.
If we take $z_t = \frac{1}{2} (t-t^2)$, then $z_t \in A \subset \frac{1}{N(x)} \mathfrak{q}_{u,v}$ satisfies $T(z_t) = t-t^2$.
In other words, $\big(tx,ty - \frac{t(t-1)}{2} N(x) \big) \in H(u,v)$.

This proves that $\pi_1(H(u,v))$ is an $A$-submodule of $L$. Note that this is a proper submodule since equation \eqref{eqn matrix n} tells us, by looking at the first column, that $y\in B$ and hence $x\in uB+B$. In particular, since $uB+B$ is finitely generated as an $A$-module, so is $\pi_1(H(u,v))$. Since $H(u,v) \cap H(L,K)^0 = H(u,v)^0$, the isomorphism of $A$-modules $H(L,K)/H(L,K)^0 \to L$ induced by $\pi_1$ provides an isomorphism $H(u,v) / H(u,v)^0 \to \pi_1(H(u,v))$.
This completes the proof of statement~2.\\

Finally, consider the composite $\varphi$ of the group homomorphism $\stab(u,v) \to \mathbb{B}$ given by $g \mapsto g_{u,v} g g_{u,v}^{-1}$ and the quotient group homomorphism $\mathbb{B} \to \mathbb{B} / \mathcal{U}_a(K) \cong \mathcal{T}(C)$.
Then $\ker \varphi = g_{u,v}^{-1} \mathcal{U}_a(K) g_{u,v} \cap \stab(u,v) = U(u,v)$.
Hence $\varphi$ induces an isomorphism between $\stab(u,v)/U(u,v)$ and a subgroup of $\mathcal{T}(C)$.
Since $\mathcal{T}(C)$ is isomorphic to $\mathbb{F}^\times$ via $\widetilde{a}$, we deduce statement~3.
\end{proof}

\begin{Rq}
In the rest of the article, we only need to know that $\pi_1(H(u,v))$ is an $\mathbb{F}$-vector space. Actually, if $\mathbb{F}$ contains at least $3$ elements, this is equivalent to $\pi_1(H(u,v))$ being an $A$-module. Indeed, there is $t_0 \in \mathbb{F}$ such that $t_0-t_0^2 \in \mathbb{F}^\times = A^\times$. Thus, the existence of $y_t$ for any $t \in A$ in the proof above becomes equivalent to the existence of $y_{t_0}$, taking $y_t= \frac{t-t^2}{t_0 - t_0^2} y_{t_0}$ by linearity of $T$ and the fact that $\frac{1}{N(x)} \mathfrak{q}_{u,v}$ is a fractional $B$-ideal. Moreover, in the particular case where $\bb F=\bb F_2$, it is simply obvious that $\pi_1(H(u,v))$ is an $\bb F$-vector space.

However, fields of characteristic 2 provide new issues to our argument. Indeed, if $\operatorname{char}(\mathbb{F})= 2$, we cannot define $z_t = \frac{t(1-t)}{2}$.
For $t \in \mathbb{F} \smallsetminus \{0,1\}$, up to dividing by $t-t^2$, the existence of a $z_t \in B$ satisfying conditions of~\eqref{equivalence H(u,v)} is equivalent to that of an element $z_1 \in B$ with $T(z_1)=1$. If such an element exists, we take $z_t = z_1 (t-t^2)$.
But in characteristic $2$, it may happen that such an element does not exist.
For instance, if $L = K(\alpha)$, where $\alpha^2+\alpha+\frac{1}{r}=0$, for some $r \in A$ with valuation $\omega(r) < 0$, one can easily show that there is no element in $B$ with trace equal to $1$. It could happen that such an element exists in $\frac{1}{N(x)} \mathfrak{q}_{u,v}$, but we do not know how much bigger than $B$ is this fractional ideal, which of course depends on $x$. This calls for further analysis, which is a reason why we have avoided fields of characteristic 2 in this article.
\end{Rq}

\begin{Def}[Stabilizers of vertices]
For $(u,v) \in H(L,K)$ and $n \in \Z_{>0}$, define the group
\[\operatorname{Stab}(u,v,n)= \operatorname{Stab}_{G(A)}\Big(\lambda_n(u,v)\Big).\]
Denote by $U(u,v,n)$ the subset of its unipotent elements and by $S(u,v,n)$ the subset of its semisimple elements.
\end{Def}

The following result gives all the information we need about these stabilizers.

\begin{Prop}\label{prop stab}
Let $(u,v) \in H(L,K)$ and $N_0=N_0(u,v):=\left\lceil\frac{2e_P}{d} \deg \widetilde{\mathfrak{q}_{u,v}}\right\rceil$.
Then:
\begin{enumerate}
\item We have $\stab(u,v,n)\subset\stab(u,v,n+1)$ for every $n>N_0$.
\item We have $\stab(u,v) = \bigcup_{n>N_0}\stab(u,v,n)$. In particular, $\stab(u,v)$ is the stabilizer in $G(A)$ of the visual limit of $\rf(u,v)$.
\item We have, for every $n>N_0$,
\[\operatorname{Stab}(u,v,n) = \Big\{ g_{u,v}^{-1} \mathrm{u}_{a}(x,y) \widetilde{a}(t) g_{u,v}\in \stab(u,v) : (x,y) \in H(L,K), \, \omega(y) \geqslant -\frac{n}{e_P}\Big\}.\]

\item There exists a non zero $B$-ideal $I=I(u,v)$ such that $I=\overline{I}$ and $H(L,K)_{I}:= H(L,K) \cap (I \times I)$ satisfies, for every $n>N_0$,
\[U(u,v,n) \supseteq \Big\{ g_{u,v}^{-1} \mathrm{u}_{a}(x,y) g_{u,v} : (x,y) \in H(L,K)_{I}, \, \omega(y) \geqslant -\frac{n}{e_P}\Big\}.\]
\end{enumerate}
\end{Prop}

\begin{proof}
In this proof, given fractional ideals $\mathfrak{b}_{i,j}$ of $B$, we denote for simplicity:
\[
    \begin{pmatrix}
        \mathfrak{b}_{1,1} & \mathfrak{b}_{1,2} & \mathfrak{b}_{1,3} \\
        \mathfrak{b}_{2,1} & \mathfrak{b}_{2,2} & \mathfrak{b}_{2,3} \\
        \mathfrak{b}_{3,1} & \mathfrak{b}_{3,2} & \mathfrak{b}_{3,3} \\
    \end{pmatrix}
    = \left\{
        \begin{pmatrix}
        {b}_{1,1} & {b}_{1,2} & {b}_{1,3} \\
        {b}_{2,1} & {b}_{2,2} & {b}_{2,3} \\
        {b}_{3,1} & {b}_{3,2} & {b}_{3,3} \\
    \end{pmatrix} \in \mathcal{M}_3(L),\
        {b}_{i,j} \in \mathfrak{b}_{i,j}\ \forall i,j
    \right\}.
\]

Clearly, we have:
\[
	\operatorname{Stab}(u,v,n)
%	= G(A) \cap \operatorname{Stab}(u,v,n)
	= G(A) \cap g_{u,v}^{-1} \operatorname{Stab}_{G(K_P)}(\lambda_n(\infty)) g_{u,v}
\]

Let $g \in G(A) \subset \mathcal{M}_3(B) = 
    \begin{pmatrix}
        B & B & B \\
        B & B & B \\
        B & B & B \\
    \end{pmatrix}$. Then
\begin{align*}
    g_{u,v} g g_{u,v}^{-1}
    & = \mathrm{u}_{-a}(u,v) \left( \mathrm{s} g \mathrm{s} \right) \mathrm{u}_{-a}(-u,\bar{v}) \\
    & \in 
        \begin{pmatrix}
            1   &       0       &   0   \\
            u   &       1       &   0   \\
            v   &   -\bar{u}    &   1   \\
        \end{pmatrix}
        \begin{pmatrix}
            B & B & B \\
            B & B & B \\
            B & B & B \\
        \end{pmatrix}
        \begin{pmatrix}
                1   &       0      &   0   \\
               -u   &       1      &   0   \\
            \bar{v} &   \bar{u}    &   1   \\
        \end{pmatrix}\\
    & \subset
        \begin{pmatrix}
            1   &       0       &   0   \\
            u   &       1       &   0   \\
            v   &   -\bar{u}    &   1   \\
        \end{pmatrix}
        \begin{pmatrix}
            \mathfrak{p}_{u,\bar{v}} & \mathfrak{b}_{\bar{u}} & B \\
            \mathfrak{p}_{u,\bar{v}} & \mathfrak{b}_{\bar{u}} & B \\
            \mathfrak{p}_{u,\bar{v}} & \mathfrak{b}_{\bar{u}} & B \\
        \end{pmatrix}\\
    & \subset
        \begin{pmatrix}
            L & L & L \\
            \mathfrak{p}_{u,\bar{v}}\mathfrak{b}_u & L & L \\
            \mathfrak{p}_{u,\bar{v}}\mathfrak{p}_{\bar{u},v} & \mathfrak{b}_{\bar{u}}\mathfrak{p}_{\bar{u},v} & L \\
        \end{pmatrix}
        =
        \begin{pmatrix}
            L & L & L \\
            \widetilde{\mathfrak{q}_{u,v}}^{-1} & L & L \\
            {\mathfrak{q}_{u,v}}^{-1} & \widetilde{\mathfrak{q}_{\bar{u},\bar{v}}}^{-1} & L \\
        \end{pmatrix}
        \\
\end{align*}

By \cite[9.3(i) and 8.10 (ii)]{Landvogt}, we know that:
\begin{equation}\label{eqn P_x}
	\operatorname{Stab}_{G(K_P)}(\lambda_n(\infty))
	= U_{-a,\lambda_n(\infty)} U_{a,\lambda_n(\infty)} N_{\lambda_n(\infty)}
\end{equation}
where
\begin{align}\label{eqn def U_a_lambda}
\begin{split}
	U_{-a,\lambda_n(\infty)}
	&= \left\{
		\mathrm{u}_{-a}(u,v),\
		\frac{1}{2} \omega(v) \geqslant -(-a)(\lambda_n(\infty)) = \frac{n}{2e_P}
	\right\},\\
	U_{a,\lambda_n(\infty)}
	&= \left\{
		\mathrm{u}_{a}(u,v),\
		\frac{1}{2} \omega(v) \geqslant -a(\lambda_n(\infty)) = -\frac{n}{2e_P}
	\right\}
\end{split}
\end{align}
and
\[
	N_{\lambda_n(\infty)}
	= \operatorname{Stab}_{\mathcal{N}_G(\mathcal{S})(K_P)}(\lambda_n(\infty))
	= \{1, \widetilde{a}(t) \mathrm{s}\} \cdot \mathcal{T}(K_P)_b
\]
with $\frac{1}{2}\omega(t) = -a(\lambda_n(\infty)) = -\frac{n}{2e_P}$ by \cite[8.6(ii), 4.21(iii) and 4.14(i)]{Landvogt}.

Now assume that $g \in \operatorname{Stab}(u,v,n)$ so that $h=g_{u,v} g g_{u,v}^{-1} \in \operatorname{Stab}_{G(K_P)}(\lambda_n(\infty))$ and write it as:
$h = \mathrm{u}_{-a}(U,V) \mathrm{u}_a(X,Y) \mathrm{n}$ for some $(U,V),(X,Y) \in H(L_P,K_P)$ and $\mathrm{n} \in \{1, \widetilde{a}(t) \mathrm{s}\} \cdot \mathcal{T}(K_P)_b$ such that:
\[
    \left\{\begin{array}{l}
     2\omega(X) \geqslant \omega(Y) \geqslant -\frac{n}{e_P}  \\
     2\omega(U) \geqslant \omega(V) \geqslant \frac{n}{e_P}
    \end{array}\right.
\]

If $\mathrm{n} \in \mathcal{T}(K_P)_b$, then write it as $\mathrm{n} = \widetilde{a}(t)$ with $t \in \mathcal{O}_P$, then we have:
\begin{align*}
    \operatorname{Mat}(h)
    &= \begin{pmatrix}
        1   &   0    &   0   \\
        U   &   1    &   0   \\
        V   &-\bar{U}&   1   \\
    \end{pmatrix}
    \begin{pmatrix}
        1   &   -\bar{X}    &   Y   \\
        0   &       1       &   X   \\
        0   &       0       &   1   \\
    \end{pmatrix}
    \begin{pmatrix}
        t   &   0   &   0\\
        0   & \bar{t}/t & 0\\
        0   &   0   &   1/\bar{t}\\
    \end{pmatrix}\\
    &=
    \begin{pmatrix}
        t   &       -(\bar{t}/t)\bar{X}        &   (1/\bar{t})Y              \\
        tU  &   (\bar{t}/t)(1-\bar{X}U)         &   (1/\bar{t})(X+UY)         \\
        tV  &   -(\bar{t}/t)(\bar{U}+V\bar{X}) &   (1/\bar{t})(1-\bar{U}X+VY)\\
    \end{pmatrix}\\
    &\in
        \begin{pmatrix}
            L & L & L \\
            \widetilde{\mathfrak{q}_{u,v}}^{-1} & L & L \\
            {\mathfrak{q}_{u,v}}^{-1} & \widetilde{\mathfrak{q}_{\bar{u},\bar{v}}}^{-1} & L \\
        \end{pmatrix}
\end{align*}
But since
\begin{align*}
    \omega(tU) &\geqslant \frac{n}{2e_P},\\
    \omega(tV) &\geqslant \frac{n}{e_P},\\
    \omega(-(\bar{t}/t)(\bar{U}+V\bar{X})) &\geqslant \min(\omega(U),\omega(V)+\omega(X)) \geqslant \frac{n}{2e_P},\\
\end{align*}
we deduce from Lemma \ref{lem Idealic} that $h \in \begin{pmatrix} L & L & L\\0&L&L\\0&0&L\end{pmatrix}$.

Otherwise $\mathrm{n} = \widetilde{a}(t) \mathrm{s}$ for some $t \in L_P$ such that $\omega(t) \geqslant \frac{-n}{e_P}$. Thus, we have:
\begin{align*}
    \operatorname{Mat}(h)
    &= \begin{pmatrix}
        1   &   0    &   0   \\
        U   &   1    &   0   \\
        V   &-\bar{U}&   1   \\
    \end{pmatrix}
    \begin{pmatrix}
        1   &   -\bar{X}    &   Y   \\
        0   &       1       &   X   \\
        0   &       0       &   1   \\
    \end{pmatrix}
    \begin{pmatrix}
        t   &   0   &   0\\
        0   & \bar{t}/t & 0\\
        0   &   0   &   1/\bar{t}\\
    \end{pmatrix}
    \begin{pmatrix}
        0   &   0   &   -1\\
        0   &  -1   &   0 \\
        -1  &   0   &   0 \\
    \end{pmatrix}\\
    &=
    \begin{pmatrix}
        -(1/\bar{t})Y               &   (\bar{t}/t)\bar{X}              & -t \\
        -(1/\bar{t})(X+UY)          &   -(\bar{t}/t)(1-\bar{X}U         & -tU \\
        -(1/\bar{t})(1-\bar{U}X+VY) &   (\bar{t}/t)(\bar{U}+V\bar{X})   & -tV \\
    \end{pmatrix}\\
    &\in
        \begin{pmatrix}
            L & L & L \\
            \widetilde{\mathfrak{q}_{u,v}}^{-1} & L & L \\
            {\mathfrak{q}_{u,v}}^{-1} & \widetilde{\mathfrak{q}_{\bar{u},\bar{v}}}^{-1} & L \\
        \end{pmatrix}
\end{align*}
But since
\begin{align*}
    \omega(-(1/\bar{t})(X+UY)) &\geqslant \frac{n}{e_P}+ \min(\omega(X),\omega(U)+\omega(Y) \geqslant \frac{n}{2e_P},\\
    \omega(-(1/\bar{t})(1-\bar{U}X+VY)) &\geqslant \frac{n}{e_P} + \min(0,\omega(U)+\omega(X),\omega(V)+\omega(Y)) \geqslant \frac{n}{e_P},\\
    \omega((\bar{t}/t)(\bar{U}+V\bar{X})) &\geqslant \min(\omega(U),\omega(V)+\omega(X)) \geqslant \frac{n}{2e_P},\\
\end{align*}
we deduce from Lemma \ref{lem Idealic} that $h \in \begin{pmatrix} L & L & L\\0&L&L\\0&0&L\end{pmatrix}$.

Hence, in both cases, we have $h \in \mathcal{B}(K) \cap \operatorname{Stab}_{G(K_P)}(\lambda_n(\infty))$
and one can write $h = \mathrm{u}_a(x,y)\widetilde{a}(t) $ for some $t \in L^\times$ and $(x,y) \in H(L,K)$ with
\[
    \left\{\begin{array}{l}
     \omega(y) \geqslant - \frac{n}{e_P}  \\
     \omega(t) = 0
    \end{array}\right.
\]
The characteristic polynomial of $g$ has coefficients in $B$ with non-negative valuation in $P$ since the valuation of its eigenvalues $(t, \bar{t}/t, 1/\bar{t})$ is equal to $0$. Hence, the characteristic polynomial of $g$ has coefficients in $\mathbb{F}$. Furthermore, as $\mathbb{F}$ is algebraically closed in $L$, then $t \in \mathbb{F}^\times$. We see then that
\[\operatorname{Stab}(u,v,n) \subseteq \Big\{ g_{u,v}^{-1} \mathrm{u}_{a}(x,y) \widetilde{a}(t) g_{u,v},\, \mathrm{u}_{a}(x,y) \widetilde{a}(t) \in \bb B, \, \omega(y) \geqslant -\frac{n}{e_P} \text{ and } \omega(x)  \geqslant -\frac{n}{2e_P} \Big\}.\]
Now we can prove 1. Let $h \in g_{u,v} G(A) g_{u,v}^{-1} \cap \text{Stab}_{G(K_P)}(\lambda_n(\infty))$. By the previous discussion, $h$ has the form $h=\mathrm{u}_a(x,y)\widetilde{a}(t)$, for some $t \in \mathbb{F}^{*}$ and $(x,y) \in H(L,K)$ with $2\omega(x) \geqslant \omega(y) \geqslant -\frac{n}{e_P}$. This implies that $h \in U_{a, \lambda_n(\infty)} \mathcal{T}(C) \subseteq U_{-a, \lambda_{n+1}(\infty)} U_{a, \lambda_{n+1}(\infty)} N_{\lambda_{n+1}(\infty)}$ by \eqref{eqn P_x}. We conclude that $h \in g_{u,v} G(A) g_{u,v}^{-1} \cap \text{Stab}_{G(K_P)}(\lambda_{n+1}(\infty))$. This proves 1.\\

We have seen that $\stab(u,v,n)\subseteq g_{u,v}^{-1}\bb Bg_{u,v}$ for any $n > N_0$.
In particular, we deduce that $\bigcup_{n > N_0} \stab(u,v,n) \subseteq g_{u,v}^{-1}\bb Bg_{u,v} \cap G(A) = \stab(u,v)$.

Conversely, let $b \in \stab(u,v)$ and write it $b = g_{u,v}^{-1} \mathrm{u}_a(x,y) \widetilde{a}(t) g_{u,v} \in G(A)$.
Then $g_{u,v} b g_{u,v}^{-1} = \mathrm{u}_a(x,y) \widetilde{a}(t)$.
We have that $\widetilde{a}(t) \in \mathcal{T}(C) \subset \mathcal{T}(K_P)_b \subset \stab_{G(K_P)}(\lambda_n(\infty))$ for any $n \in \Z_{>0}$ since $\mathcal{T}(K_P)_b$ fixes the standard apartment.
Moreover, for $n_1 > N_0$ large enough, we have that $\omega(y) > - \frac{n_1}{e_P}$.
Hence $\mathrm{u}_a(x,y) \in \stab_{G(K_P)}(\lambda_{n_1}(\infty))$.
Thus, $\mathrm{u}_a(x,y) \widetilde{a}(t) \in \stab_{G(K_P)}(\lambda_{n_1}(\infty))$ and, since $b \in G(A)$, we deduce that $b \in \stab(u,v,n_1)$.
Hence we have $\stab(u,v) \subseteq \bigcup_{n > N_0} \stab(u,v,n)$, and therefore the equality of 2. This equality clearly implies that $\stab(u,v)$ stabilizes the visual limit of $\rf(u,v)$. On the other hand, if $g\in G(A)$ stabilizes the visual limit of $\rf(u,v)$, then the element $g_{u,v}gg_{u,v}^{-1}$ stabilizes the visual limit of $\rf(\infty)$, hence it is contained in $\cal B(K)$. We can write $g_{u,v}gg_{u,v}^{-1}=\mathrm{u}_a(x,y)\widetilde{a}(t)$, for some $t \in L^\times$ and $(x,y) \in H(L,K)$. Then, the eigenvalues of $g$ are $t$, $\bar{t}/t$ and $1/\bar{t}$, which belong to $L$. On the other hand, the characteristic polynomial of $g \in \mathcal{M}_3(B)$ belongs to $B[x]$, and therefore $t$, $\bar{t}/t$ and $1/\bar{t}$ are integral over $B$. Since $B$ is integrally closed, we get $t,\bar{t}/t, 1/\bar{t} \in B$. In particular, we deduce that $\omega(t), \omega(1/t)=\omega(1/\bar{t}) \leq 0$. This implies that $\omega(t)=0$, whence $t \in \mathcal{O}_Q \cap B=\mathbb{F}$. We conclude that $g_{u,v}gg_{u,v}^{-1} \in \mathbb{B}$, so that $g \in \mathrm{Stab}(u,v)$. This proves 2.\\

We claim now that all elements in $\stab(u,v,n+1)\smallsetminus\stab(u,v,n)$ are of the form $g_{u,v}^{-1} \mathrm{u}_{a}(x,y) \widetilde{a}(t) g_{u,v}$ with $\omega(y)=-\frac{n+1}{e_P}$. Indeed, 
for any $g\in \operatorname{Stab}(u,v,n)$ we have $h=g_{u,v} g g_{u,v}^{-1} \in \text{Stab}_{G(K_P)}(\lambda_n(\infty))$. So the uniqueness in the decomposition of any element of 
$\text{Stab}_{G(K_P)}(\lambda_n(\infty))$ as a product in $ U_{-a,\lambda_n(\infty)} U_{a,\lambda_n(\infty)} N_{\lambda_n(\infty)}$,which follows from Bruhat decomposition, shows the claim. This implies 3.\\

Finally, let us prove 4. Consider the matrix $\mathfrak{n}$ given in \eqref{eqn matrix n}, in the proof of Lemma~\ref{Lemma M}.
Observe that all entries of $\mathfrak{n}$ are linear polynomials in the variables $x$, $\bar{x}$ and $y$.
Moreover, if we let $p_{ij}(x,y)=a_{ij} x+ b_{ij} y + c_{ij} \bar{x}$ be the polynomial corresponding to the entry $(i,j)$ of $\mathfrak{n}$, then for any $i,j \in \lbrace 1,2,3\rbrace$ we can write $a_{ij}=\frac{a'_{ij}}{a''_{ij}}$, $b_{ij}=\frac{b'_{ij}}{b''_{ij}}$ and $c_{ij}=\frac{c'_{ij}}{c''_{ij}}$, where $a'_{ij}, a''_{ij},b'_{ij}, b''_{ij},c'_{ij}, c''_{ij} \in B$. Let $K_{ij}$ be the principal $B$-ideal generated by $a''_{ij}b''_{ij}c''_{ij}$.
Then $\mathfrak{n} \in \mathcal{M}_3(B)$ if and only if $a'_{ij} x+ b'_{ij} y + c'_{ij} \bar{x} \in K_{ij}$ for each $i,j \in \lbrace 1,2,3\rbrace$. So, the previous condition holds for $x,y \in \bigcap_{i,j=1}^3 (K_{ij} \cap \overline{K_{ij}})$. In particular we can take $I=\bigcap_{i,j=1}^3 (K_{ij} \cap \overline{K_{ij}}) \subset B$, which clearly satisfies $I=\overline{I}$. This proves 4.
\end{proof}

\section{Riemann-Roch and cusps in \texorpdfstring{$\overline X$}{X}}\label{sec RR}

In this section, we apply the Riemann-Roch Theorem in order to compare stabilizers of the different $\lambda_n(u,v)$. Let us recall first this statement and apply it to our content.

\paragraph*{Riemann-Roch Theorem.}
Let $J \subset L$ be a proper non-zero fractional $B$-ideal, which we see as a line bundle associated to a divisor $D_J$ in the affine curve $\spec(B)=D\smallsetminus\{Q\}$, and let $m \in \Z_{>0}$. Note that $\deg(D_J)=-\deg(J)$. Recalling that $\Gamma= \omega(L_P^{\times})= \frac{1}{e_P} \mathbb{Z}$, define
\[J[m]:=\cal L(D_J+mQ)=\left\lbrace x \in J:\omega(x) \geq -\frac{m}{e_P}\right\rbrace.\]
We denote by $g_D$ the genus of $D$. Then by the Riemann-Roch Theorem \cite[\S 1, Thm.~1.5.17]{Stichlenoth} the set $J[m]$ is a finite-dimensional vector space over $\mathbb{F}$, and when $\deg(D_J+mQ)\geq 2g_D-1$ we have
\[\text{dim}_{\mathbb{F}}(J[m])= \deg(D_J+mQ)+ 1-g_D.\]
Let $f_P=[\kappa_Q:\kappa_P]$ be the residual degree of $L/K$ at $P$. Hence
\[\deg(D_J+mQ)=-\deg(J)+m\deg(Q)=-\deg(J)+mf_P\deg(P)=-\deg(J)+mf_Pd,\]
so that we finally get,
\begin{equation}\label{eq rr0}
\text{dim}_{\mathbb{F}}(J[m])= -\deg(J)+mf_Pd+ 1-g_D,\quad\text{ when }\quad mf_Pd\geq \deg(J)+2g_D-1.
\end{equation}

This result will help us to define the ``legs of the spider'', that is, the cusps in the quotient graph $\overline X$. In order to do this, we use the rays $\rf(u,v)$, but starting ``away enough''. The first result of this section fixes such a bound. But first we need some definitions.

\begin{Def}
We denote by $M(u,v):=\pi_1(g_{u,v} U(u,v)g_{u,v}^{-1})$ the $A$-submodule of $L$ defined in Lemma \ref{Lemma M}.2. Note that it contains the $B$-ideal $I(u,v)$ by Proposition \ref{prop stab}.4. Similarly, we define
\[M(u,v,n):=\pi_1(g_{u,v} U(u,v,n)g_{u,v}^{-1}).\]
Finally, define $r(u,v)$ as the $\bb F$-dimension of the quotient $M(u,v)/I(u,v)$.
\end{Def}

Note that $M(u,v,n)$ contains $I(u,v)[\frac n2]$. Indeed, for each $x \in I(u,v)[\frac n2]$ we have $(x,y)\in H(L,K)_I$ with $y:=-N(x)/2$ and $\omega(y)=\omega(N(x)) \geq -\frac{n}{e_P}$. So, it follows from Proposition \ref{prop stab}.4 that $g_{u,v}^{-1} \mathrm{u}_a(x,y) g_{u,v} \in U(u,v,n)$, whence $x \in \pi_1(g_{u,v} U(u,v,n)g_{u,v}^{-1})$.

Note also that $M(u,v,n)$ is an $\bb F$-vector space. Indeed, let $x\in M(u,v,n)$ and let $y \in L$ be such that $g_{u,v}^{-1} \mathrm{u}_a(x,y) g_{u,v} \in U(u,v,n)$. For any $\lambda \in \mathbb{F}$, the pair $(\lambda x, \lambda^2 y)$ belongs to $H(L,K)$ and satisfies $\omega(\lambda^2 y)=\omega(y) \geq -\frac{n}{e_P}$. Thus, we get $g_{u,v}^{-1} \mathrm{u}_a(\lambda x, \lambda^2 y) g_{u,v} \in U(u,v,n)$ and $\lambda x\in M(u,v,n)$.

\begin{Lem}\label{lemma ideals}
The dimension $r(u,v)$ is finite. Moreover, there exists $N_1=N_1(u,v)\in\Z_{>0}$ such that, for $n\geq N_1$, $M(u,v,n)/I(u,v)[\frac n2]$ has dimension $r(u,v)$.
\end{Lem}
\begin{proof}
Let $P(u,v)$ be the proper $B$-fractional ideal generated by $M(u,v) \subset L$. Then, since $B$ is a Dedekind domain, $P(u,v)/I(u,v)$ is a finite-dimensional $\F$-vector space. In particular, $M(u,v)/I(u,v)$ is a finite dimensional $\mathbb{F}$-vector space, which proves the finiteness of $r(u,v)$.

This implies that there exists a set $\lbrace a_i= a_i(u,v) \rbrace_{i=1}^{r}$ such that $M(u,v)=I(u,v) \oplus \bigoplus_{i=1}^r a_i \mathbb{F}$. For each $i \in \lbrace 1, \cdots, r\rbrace$, let $b_i \in L$, so that $ g_{u,v}^{-1}\mathrm{u}_a(a_i,b_i) g_{u,v}\in U(u,v)$.
Define then
\[N_1:=\max\lbrace N_0(u,v), \max \lbrace -e_P \cdot\omega(b_i): i \in \{1,\ldots,r\} \rbrace\rbrace.\]
Any $\sigma \in M(u,v,n)$ can be written as $\sigma= a + a_1f_1+ \cdots+ a_m f_m$, where $f_i \in \mathbb{F}$ and $a \in I(u,v)$. Let $\tau \in L$ such that $ g_{u,v}^{-1}\mathrm{u}_a(\sigma,\tau) g_{u,v}\in U(u,v,n)$. If $n> N_1$, by definition, we have $g_{u,v}^{-1}\mathrm{u}_a(a_i,b_i) g_{u,v}\in U(u,v,n)$, hence $a_i \in M(u,v,n)$ for all $i$, and thus $a \in M(u,v,n)$. Therefore, there exists $b \in L$ such that $g_{u,v}^{-1}\mathrm{u}_a(a,b) g_{u,v}\in U(u,v,n)$ and $\omega(b) \geq -\frac{n}{e_P}$. Then $\omega(a) = \frac{1}{2} \omega(\tr(b)) \geq \frac{1}{2} \omega(b) \geq -\frac{n}{2e_P}$, whence $a \in I(u,v)[\frac n2]$. The result follows.
\end{proof}

\begin{Def}
For $(u,v)\in H(L,K)$, define
\[N(u,v):= \max\left\lbrace  \frac{2}{f_Pd}(\deg(I(u,v))+2g_D-1), N_1(u,v)\right\rbrace.\]
We define $\rf_0(u,v)$ as the subray $\lbrace \lambda_{n}(u,v):n>N(u,v)\rbrace\subset\rf(u,v)$ of $\rf(u,v)$ and $\cf(u,v) \subset \overline{X}$ as the image of $\rf_0(u,v)$ in $\overline{X}$ via the canonical projection.

For $x=\lambda_n(u,v)\in \rf(u,v)$, we denote by $r_{x}$ the edge defined by $x=\lambda_n(u,v)$ and $\lambda_{n+1}(u,v)$.
\end{Def}

One must think of $r_x$ as the edge ``going outwards''. This will allow us to see the cusps as directed rays in some sense, which will come in handy when proving that they are actual rays in $\overline{X}$.

We start by using Riemann-Roch in order to bound the valency of the vertices in the cusps.

\begin{Lem}\label{lemma valency cusps}
Let $x\in\rf_0(u,v)$ be a lift of a vertex in $\cf(u,v)$. Then all edges different from $r_x$ lie on the same $G(A)$-orbit. In particular, every vertex of the subgraph $\cf(u,v)$ of $\overline{X}$ has valency at most two.
\end{Lem}

\begin{proof}
For any $x \in v(X)$ we define $\mathcal{V}^1(x)$ as the star of $x$, i.e, the subcomplex of $X$ whose vertices are $x$ and its neighbors. Let $P_x=\text{Stab}_{G(K_P)}(x)$ and $P_x^{*}=\bigcap_{w \in v(\mathcal{V}^1(x))} P_w$. The group $P_x^{*}$ is a normal subgroup of $P_x$. We denote by $\pi_x: P_x \rightarrow G_x= P_x/P_x^{*}$ the canonical projection.

Let $x_n:=\lambda_n(u,v)\in \rf(u,v)$ and write $y_n:=\lambda_n(\infty)$ so that $x_n=g_{u,v}^{-1}y_n$. By \cite[9.7(i)]{Landvogt}, we know that for every such $x$ (resp.~$y$) there is a subgroup $U_{r_x}$ of $P_x$ (resp.~$U_{r_y}$ of $P_y$), generated by unipotent elements, acting transitively on the set of apartments containing $r_x$ (resp.~$r_y$). Hence, the image $\bar U_{r_x}$ of $U_{r_x}$ in $G_x$ acts transitively on the sets of edges $r\neq r_x$ joined to $x$. We only need to prove then that the image in $G_{x_n}$ of unipotent elements in $\stab(u,v,n)$ generate $\bar U_{r_{x_n}}$, which amounts to proving that the image in $G_{y_n}$ of $g_{u,v} U(u,v,n) g_{u,v}^{-1}$ generates $\bar U_{r_{y_n}}$.

By \cite[Def~8.8]{Landvogt}, we have that
\[U_{r_{y_n}}=\langle U_{-a,\lambda_{n+1}(\infty)},U_{a, \lambda_n(\infty)}\rangle\quad \text{and}\quad P_{y_n}^*\supseteq \langle U_{-a, \lambda_{n+1}(\infty)}, U_{a, \lambda_{n-1}(\infty)}\rangle,\]
where these groups were defined in \eqref{eqn def U_a_lambda}. So we only need to prove that $g_{u,v} U(u,v,n) g_{u,v}^{-1}$ covers the quotient $U_{a, \lambda_n(\infty)}/U_{a, \lambda_{n-1}(\infty)}$.

Consider then $\mathrm{u}_a(x,y) \in U_{a,\lambda_{n}(\infty)}$ with $\omega(y)=-\frac{n}{e_P}$. Assume first that $\omega(\tr(y))>\omega(y)$ and define $y':=y-\frac 12\tr(y)$. We have then $\omega(y')=\omega(y)$ and $\tr(y')=0$. We claim that we can write $y'$ as $y'=v_0+v_1$, with $v_0\in I=I(u,v)$, $\omega(v_0)=\omega(y')=\omega(y)$ and $\omega(v_1)>\omega(v_0)$, where $I(u,v)$ is defined in Proposition \ref{prop stab}. Indeed, 
under the condition $-\omega(y)>N(u,v)$, the quotient $I[-e_P\omega(y)]/I[1-e_P\omega(y)]$ has the same $\bb F$-dimension as $\kappa_Q$, as it can be deduced from equation \eqref{eq rr0}, hence such an element $v_0$ always exists. Moreover, we may assume that $\tr(v_0)=0$. Indeed, since $\tr(y')=0$, we conclude that $\tr(v_0)=-\tr(v_1)$. And since $\omega(-\tr(v_1))\geq\omega(-v_1)>\omega(v_0)$, we can replace
$v_0$ by $v_0-\frac12\tr(v_0)$ and get our claim since $\tr(v_0)\in I(u,v)$ (recall that $I=\overline{I}$).

We have then that $\mathrm{u}_a(0,v_0)\in g_{u,v} U(u,v,n) g_{u,v}^{-1}$ and, if we set $v_2=v_1+\frac12\tr(y)$, we get the matrix equality $\mathrm{u}_{a}(x,y)=\mathrm{u}_{a}(x,v_2)\mathrm{u}_{a}(0,v_0)$. Since both $\mathrm{u}_{a}(x,y)$ and $\mathrm{u}_{a}(0,v_0)$ are in $\cal U_a$, then so does $\mathrm{u}_{a}(x,v_2)$ and thus $N(x)+\tr(v_2)=0$. And since both $w_1$ and $\tr(v_1)$ have valuation greater than $\omega(y)$, we get that $\mathrm{u}_{a}(x,v_2)\in U_{a, \lambda_{n-1}(\infty)}$, which concludes the argument in this case.

Assume finally that $\omega(\tr(y))=\omega(y)$. Using once again the Riemann-Roch Theorem, we can write $x$ as $x=u_0+u_1$, with $u_0\in I=I(u,v)$, $\omega(u_0)=\omega(x)$ and $\omega(u_1)>\omega(x)$. Define $v_0:=-\frac 12u_0\overline{u_0}$, so that $v_0\in I$ (recall that $I$ is an ideal of $B$), $\omega(v_0)=2\omega(u_0)=2\omega(x)$ and $N(u_0)+\tr(v_0)=0$. In particular, we may consider the element $\mathrm{u}_{a}(u_0,v_0)\in \cal U_a$, which is actually in $g_{u,v} U(u,v,n) g_{u,v}^{-1}$ by Proposition \ref{prop stab}.4. Finally, define $v_1:=y+u_0\overline{u_1}-v_0$, so that we get the matrix equality
$\mathrm{u}_{a}(x,y)=\mathrm{u}_{a}(u_1,v_1)\mathrm{u}_{a}(u_0,v_0)$. In particular, since both $\mathrm{u}_{a}(x,y)$ and $\mathrm{u}_{a}(u_0,v_0)$ are in $\cal U_a$, we see that so does $\mathrm{u}_{a}(u_1,v_1)$ and thus $N(u_1)+\tr(v_1)=0$. If $\omega(v_1)>\omega(y)$, then $\mathrm{u}_{a}(u_1,v_1)\in U_{a, \lambda_{n-1}(\infty)}$ and we are done. Otherwise, since by construction we have $\omega(u_0\overline{u_1})>\omega(v_0)=\omega(y)$, we get $\omega(v_1)\geq\omega(y)$ and thus we may assume $\omega(v_1)=\omega(y)$.
Now, since $\omega(u_1)>\omega(x)$, we see that
\[\omega(\tr(v_1))=\omega(N(u_1))>\omega(N(x))=\omega(\tr(y))=\omega(y)=\omega(v_1).\]
So that up to replacing our element $\mathrm{u}_a(x,y) \in U_{a,\lambda_{n}(\infty)}$ by $\mathrm{u}_a(u_1,v_1)$, we are reduced to the first case, which concludes the proof.
\end{proof}

We would like to prove now that $\cf(u,v)$ is actually a ray in $\overline{X}$. In order to do this, we apply Lemma \ref{lemma ideals} (which depends on Riemann-Roch) as follows.

\begin{Lem}\label{lem c(u,v) is a ray}
If $\lambda_n(u,v)$ and $\lambda_m(u',v')$ are on the same $G(A)$-orbit with $n>N(u,v)$ and $m>N(u',v')$, then
\[f_P d \left\lfloor \frac{n}{2}\right\rfloor - \deg \big( I (u,v)\big) +r(u,v)= f_Pd \left\lfloor \frac{m}{2}\right\rfloor -\deg \big( I(u',v') \big)+r(u',v').\]
In particular, $\rf_0(u,v)$ does not have two vertices in the same $G(A)$-orbit, i.e.~$\cf(u,v)$ is a ray.
\end{Lem}

\begin{proof}
Suppose that there exists $g \in G(A)$ such that $g\cdot\lambda_n(u,v)=\lambda_m(u',v')$. Then the element $h = g_{u',v'} g g_{u,v}^{-1}\in G(K)$ satisfies that $h U(u,v,n) h^{-1}=U(u',v',m)$. On the other hand, by Lemma \ref{lemma ideals}, we have
\begin{equation}\label{eq dim T and I}
\dim_\mathbb{F}(M(u,v,n)) = \dim_{\mathbb{F}}(I(u,v)[\textstyle{\frac n2}])+r(u,v),
\end{equation}
and an analogous equality holds for $M(u',v',m)$. So, it follows from the hypothesis on $m$ and $n$ and the Riemann-Roch Theorem (cf.~equation \eqref{eq rr0}) that $M(u,v,n)$ and $M(u',v',m)$ are both nontrivial. In particular, there are upper triangular matrices in $U(u,v,n)$ with nontrivial coordinates above the diagonal. Let $g_1=\mathrm{u}_{a}(u_1,v_1) \in  U(u,v,n)$, with $u_1, v_1 \neq 0$, and write $\text{Mat}(h)=(a_{ij})_{i,j=1}^3$. So, it follows from $h U(u,v,n) h^{-1}=U(u',v',m)$, that there exists $g_2=\mathrm{u}_{a}(u_2,v_2) \in  U(u',v',m)$ such that
$$
\textnormal{\footnotesize$\begin{pmatrix}
        a_{11}    &  a_{12}  &   a_{13}  \\
        a_{21}   &   a_{22}   & a_{23}    \\
        a_{31}   & a_{32} &   a_{33} \\
    \end{pmatrix}$\normalsize} \textnormal{\footnotesize$\begin{pmatrix}
        1    &  -\overline{u_1}  &   v_1 \\
        0   &   1   & u_1    \\
        0   & 0 &  1 \\
    \end{pmatrix}$\normalsize} =
    \textnormal{\footnotesize$\begin{pmatrix}
        1    &  -\overline{u_2}  &   v_2 \\
        0   &   1   & u_2    \\
        0   & 0 &  1 \\
    \end{pmatrix}$\normalsize} \textnormal{\footnotesize$\begin{pmatrix}
        a_{11}    &  a_{12}  &   a_{13}  \\
        a_{21}   &   a_{22}   & a_{23}    \\
        a_{31}   & a_{32} &   a_{33} \\
    \end{pmatrix}$\normalsize},
$$
whence we have 
$$
    \textnormal{\footnotesize$\begin{pmatrix}
        a_{11}    &  a_{12}-\overline{u_1}a_{11}  &   a_{13}+u_1 a_{12}+v_1 a_{11}  \\
        a_{21}   &   a_{22}-\overline{u_1}a_{21}   & a_{23}+ u_1 a_{22}+ v_1 a_{21}   \\
        a_{31}   & a_{32}-\overline{u_1}a_{31} &   a_{33} + u_1 a_{32}+ v_1 a_{31} \\
    \end{pmatrix}$\normalsize}=
    \textnormal{\footnotesize$\begin{pmatrix}
        a_{11}- \overline{u_2} a_{21}+ v_2 a_{31}    &   a_{12}- \overline{u_2} a_{22}+ v_2 a_{32} &    a_{13}- \overline{u_2} a_{23}+ v_2 a_{33}  \\
        a_{21}+ u_2 a_{31}   &   a_{22}+ u_2 a_{32}   & a_{23}+ u_2 a_{33}    \\
        a_{31}   & a_{32} &   a_{33}\\
    \end{pmatrix}$\normalsize}.
$$
Thus, if we analyze the coordinates in the last row of the previous matrices, we deduce $a_{31}=0$, and then $a_{32}=0$. Note that, if $u_2=0$ (and hence $v_2\neq 0$), then, it follows from the previous equation, by comparing the second columns of both matrices, that $a_{11}=a_{21}=0$, which contradicts the fact that $h$ is invertible. Thus, we may assume that $u_2 \neq 0$. We deduce then, by comparing the first columns of the previous matrices, that $a_{21}=0$. Finally, we conclude that $\text{Mat}(h)$ is a upper triangular matrix. In other words, $h=\widetilde{a}(t)\mathrm{u}_{a}(z,w)$, for some $t \in L^{\times}$ and some $(z,w) \in H(L,K)$. Hence, for $\mathrm{u}_{a}(x,y)\in U(u,v,n)$ we have
\[h\mathrm{u}_{a}(x,y) h^{-1}= \mathrm{u}_{a}\left(\overline{t}^2/t x, t\overline{t}(y+\overline{x}z-\overline{z}x)\right).\]
This implies that $M(u',v',m) = \lambda M(u,v,n)$, where $\lambda = \overline{t}^2/t \in L^{\times}$. In particular, $\dim_{\mathbb{F}}(M(u',v',m)) = \dim_{\mathbb{F}}(M(u,v,n))$ and thus equation \eqref{eq dim T and I} together with the Riemann-Roch Theorem imply that
\begin{multline*}
f_P d \left\lfloor \frac{n}{2}\right\rfloor + 1-\text{deg}\big( I(u,v)\big)-g +r(u,v)=f_P d\left\lfloor \frac{m}{2}\right\rfloor+ 1-\text{deg}\big(I(u',v'))-g +r(u',v'),
\end{multline*}
and the result follows.
\end{proof}

The next result builds on the last two and gives a criterion for the equivalency of two rays $\cf(u,v)$ and $\cf(u',v')$. Moreover, this result shows that the intersection of non-equivalent rays on $\overline{X}$ is a finite subgraph of $\overline{X}$. Obviously, this results applies to the image in $\overline{X}$ of $\rf(\infty)$ too, by replacing $\rf(0,0)$ by $\rf(\infty)$.

\begin{Prop}\label{prop link}
Suppose that $\lambda_n(u,v)$ and $\lambda_m(u',v')$, are in the same $G(A)$-orbit for some $n > N(u,v)$ and $m >N(u',v')$. Then $\lambda_{n+t}(u,v)$ and $\lambda_{m+t}(u',v')$ are in the same $G(A)$-orbit, for all $t \in \mathbb{Z}_{\geq 0}$.
\end{Prop}

\begin{proof} We denote by $\overline{w}$ the image of $w \in v(X)$ in $\overline{X}$. By an inductive argument we can reduce our proof to the case $t=1$. By Lemma \ref{lemma valency cusps}, if $\overline{\lambda_n(u,v)}=\overline{\lambda_m(u',v')}$ then $\overline{\lambda_{n+1}(u,v)} \in \left\lbrace  \overline{\lambda_{m+1}(u',v')},\overline{\lambda_{m-1}(u',v')} \right\rbrace $. Assume that $\overline{\lambda_{n+1}(u,v)}=\overline{\lambda_{m-1}(u',v')}$, then by Lemma \ref{lem c(u,v) is a ray} we have that
\[f_P d \left\lfloor \frac{m}{2}\right\rfloor- \deg (I (u,v) )  +r(u',v') = f_P  d \left\lfloor \frac{n}{2}\right\rfloor- \deg I(u',v') +r(u,v),\]
and 
\[
f_P  d \left\lfloor \frac{m-1}{2}\right\rfloor- \deg (I (u,v) ) +r(u',v') =  f_P  d \left\lfloor \frac{n+1}{2}\right\rfloor  - \deg ( I(u',v') ) +r(u,v), \]
whence we get a contradiction. Hence $\overline{\lambda_{n+1}(u,v)}=\overline{\lambda_{m+1}(u',v')}$, and the result follows.
\end{proof}

As a consequence, we get the following result, which will allow us to define the cusps in the quotient graph, that is, the ``legs of the spider''.

\begin{Cor}\label{cor legs}
Suppose that $\cf(u,v)$ is equivalent to $\cf(u',v')$. Then either $\cf(u,v)$ is a subgraph of $\cf(u',v')$ or the converse.
\end{Cor}

\begin{proof}
Since the two rays are equivalent, we know that they contain at least one common vertex $\bar x$, i.e.,  $\bar x=\overline{\lambda_n(u,v)}=\overline{\lambda_m(u',v')}$, with $n>N(u,v)$ and $m>N(u',v')$. By Proposition \ref{prop link}, we know that there is a ray in $\overline{X}$ starting at $\bar x$ that is contained both in $\cf(u,v)$ and $\cf(u',v')$. If $\bar x$ is the terminal vertex of either $\cf(u,v)$ or $\cf(u',v')$, then we are done. Otherwise, by Lemma \ref{lemma valency cusps}, we know that $x$ has valency 2 and hence $\cf(u,v)$ and $\cf(u',v')$ must coincide on the preceding vertex, that is $\overline{\lambda_{n-1}(u,v)}=\overline{\lambda_{m-1}(u',v')}$. Arguing by induction, we may assume that $\bar x$ is the terminal vertex of at least one of the two rays, which concludes the proof.
\end{proof}

\begin{Rq}
Let $\partial_\infty(X)_\mathrm{rat}$ be the set of visual limits of the $\rf(u,v)$ for $(u,v) \in H(L,K) \cup \{\infty\}$. The set of visual limits $\partial_\infty(X)$ of $X$ (which corresponds to the spherical building associated to $(G,K_P)$) may be identifed with $H(L_Q,K_P) \cup \{\infty\}$. With this identification, $\partial_\infty(X)_\mathrm{rat}$ identifies with $H(L,K) \cup \{\infty\}$ which is a strict subset.
Thus, there are rays $g \cdot \rf(\infty)$ for $g \in G(K_P) \smallsetminus G(K)$ whose visual limit is in  $\partial_\infty(X) \smallsetminus \partial_\infty(X)_\mathrm{rat}$.
With our method, which follows Mason's method for $\SL_2$, we cannot say anything about the behaviour of these rays. In particular, we cannot exclude that their image in $\overline{X}$ is eventually a cusp ray contained in the ``body of the spider'', providing ``irrational cusps'' inside it. This was discarded for $\SL_2$ by Serre using deeper results on vector bundles of rank $2$.
\end{Rq}

\section{Lattice interpretation of vertices and edges}\label{sec lattices}

In this section, we use a rather well-known description of $X$ in terms of three dimensional lattices (c.f.~for instance \cite[Ex.~3.11]{Tits} or \cite{BT-SU}) to study the set of neighbors of a fixed vertex in $X$. This point of view will be useful when dealing with examples in \S\ref{examples} and in some proofs in \S\ref{sec number}.\\

Following \cite[\S 2]{C}, we define the Bruhat-Tits tree for $\mathrm{SU}(h)$ in terms of the Bruhat-Tits building for $\mathrm{SL}_3$. One of the two types of vertices of the tree are the vertices in the building that correspond to the homothety classes $[M]$  of unimodular lattices $M$, where $[\Lambda]$ denotes the class of a lattice $\Lambda$,
while the other corresponds to pairs of neighboring classes of
the form $([\Lambda],[\hat{\Lambda}])$, where $\hat{\Lambda}=\{x\in L^3\mid h(x,\Lambda)\subseteq B\}$. Inside a (suitable) fixed apartment, the vertex corresponding to a pair $([\Lambda],[\hat{\Lambda}])$ can be seen as the middle point of the edge joining the vertices corresponding to $[\Lambda]$ and $[\hat{\Lambda}]$, as it is shown in Figure \ref{Figure building SL3}. Note that this construction is related to a more general result by Prasad and Yu in \cite{Prasad-Yu}.

\begin{figure}[h!]
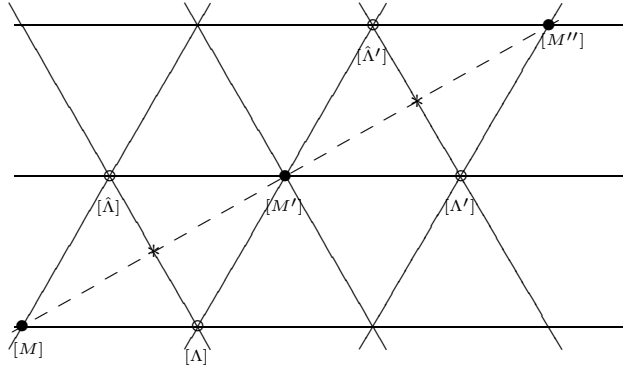

$$
\xygraph{
!{<0cm,0cm>;<1cm,0cm>:<0cm,1cm>::}
!{(-0.2,0) }*+{}="a1"
!{(-0.2,2) }*+{}="a2"
!{(-0.2,4) }*+{}="a3"
!{(8,0) }*+{}="b1"
!{(8,2) }*+{}="b2"
!{(8,4) }*+{}="b3"
!{(-.231,-.4) }*+{}="c1"
!{(2.079,-.4) }*+{}="c2"
!{(4.381,-.4) }*+{}="c3"
!{(2.541,4.4) }*+{}="d1"
!{(4.851,4.4) }*+{}="d2"
!{(7.161,4.4) }*+{}="d3"
!{(-.231,4.4) }*+{}="e1"
!{(2.079,4.4) }*+{}="e2"
!{(4.381,4.4) }*+{}="e3"
!{(2.541,-.4) }*+{}="f1"
!{(4.851,-.4) }*+{}="f2"
!{(7.161,-.4) }*+{}="f3"
!{(-.231,-.133) }*+{}="g1"
!{(7.161,4.133) }*+{}="g2"
!{(0,0) }*+{\bullet}="h1"
!{(3.465,2) }*+{\bullet}="h2"
!{(6.93,4) }*+{\bullet}="h3"
!{(.1,-.3) }*+{{}_{{}_{[M]}}}="h1n"
!{(3.465,1.6) }*+{{}_{{}_{[M']}}}="h2n"
!{(7.13,3.8) }*+{{}_{{}_{[M'']}}}="h3n"
!{(1.7325,1) }*+{*}="j1"
!{(5.1975,3) }*+{*}="j2"
!{(1.155,2) }*+{\circ}="k1"
!{(2.31,0) }*+{\circ}="k2"
!{(4.62,4) }*+{\circ}="k3"
!{(5.775,2) }*+{\circ}="k4"
!{(1.155,1.6) }*+{{}_{{}_{[\hat{\Lambda}]}}}="k1n"
!{(2.31,-.4) }*+{{}_{{}_{[\Lambda]}}}="k2n"
!{(4.62,3.6) }*+{{}_{{}_{[\hat{\Lambda}']}}}="k3n"
!{(5.775,1.6) }*+{{}_{{}_{[\Lambda']}}}="k4n"
"a1"-"b1" "a2"-"b2" "a3"-"b3"
"c1"-"d1" "c2"-"d2" "c3"-"d3"
"e1"-"f1" "e2"-"f2" "e3"-"f3"
"g1"-@{--}"g2"
} 
$$
\caption{The building for the unitary group inside the building for $\mathrm{SL}_3$.}\label{Figure building SL3}
\end{figure}

This gives a simple way to describe the neighbors of each vertex:

\subparagraph{Neighbors of a unimodular lattice.} Consider the class of a local unimodular lattice $M$. This is the case for instance of the special vertex $\lambda_0(\infty)$, which corresponds to the class of the lattice $\Lambda_0=\mathcal{O}_Q^3$. The neighbors of $[M]$ are the pairs $([\Lambda]$,$[\hat{\Lambda}])$, where $\Lambda$ is a maximal proper integral sublattice of $M$ and $\hat{\Lambda}=M+\mathcal{O}_Q\pi_Q^{-1}v$, with $v\in M$ primitive. The condition that $\Lambda$ and $\hat{\Lambda}$ are neighbors translates as $\pi_Q$ dividing $h(v,v)$.

Recall that, when the cover $\psi:D\rightarrow C$ ramifies at $Q$, we have $\bar x\equiv x\ (\mathop{\mathrm{mod}}\pi_Q)$ for $x\in\cal O_Q$. Therefore, if we denote by $M_Q$ the quotient 
$M/\pi_Q M\cong\kappa_Q^3$, the induced form $h_Q:M_Q\times M_Q\to\kappa_{Q}$ is a non-singular symmetric bilinear form. When $\psi$ is inert at $P$, the same construction defines a hermitian form over the residue field. In either case, the neighbors of the lattice $M$ are in correspondence with the isotropic lines in the space $M_Q$.

\subparagraph{Neighbors of a pair of lattices.}
Every hermitian lattice has a Jordan decomposition, i.e.~an orthogonal decomposition of the form
\[\Lambda=\Lambda_m\perp\Lambda_{m+1}\perp\cdots\perp\Lambda_n,\]
for some integers $m\leq n$, where each lattice $\Lambda_i$ is $\pi_Q^{i}$-modular, in the sense that $h(v,\Lambda_i)=\pi_Q^{i}\mathcal{O}_Q$ for every primitive vector $v\in\Lambda_i$. Note that
$\pi_Q\Lambda_i$ is $\pi_Q^{i+2}$-modular. A simple computation shows that 
$$\hat{\Lambda}=\pi_Q^{-m}\Lambda_m\perp\pi_Q^{-m-1}\Lambda_{m+1}\perp\cdots\perp\pi_Q^{-n}\Lambda_n.$$
In particular, two mutually dual classes $[\Lambda]$ and  $[\hat{\Lambda}]$ are neighbors in the building  of $\SL_3$ if and only if each is a sum of precisely two consecutive Jordan components. Multiplying by a scalar and replacing $\Lambda$ by $\hat{\Lambda}$ if needed, we might assume $\Lambda=\Lambda_0\perp\Lambda_1$. By discriminant considerations, we conclude that $\Lambda_0$ has rank $1$ and $\Lambda_1$ has rank $2$. For such a pair, its unimodular neighbors in the tree of $\su(h)$ have the form $\Lambda+\mathcal{O}_Q\pi_Q^{-1}v$, where $v\in \Lambda_1$ satisfies $h(v,v)\in\pi_Q^2\mathcal{O}_Q$. Up to a re-scaling, i.e.~replacing $h$ by $\pi_Qh$, we can again reduce the problem to the study of isotropic lines in a two dimensional $\kappa_Q$-space $\Lambda_{1,Q}=\Lambda_1/\pi_Q\Lambda_1$. Note however that the re-scaling must be done carefully since we still want to deal with a hermitian or a skew-hermitian form. Therefore, we assume that $\pi_Q\in K_P$ in the unramified case, while we choose $\pi_Q$ of trace $0$ in the ramified case. In the latter case, $\pi_Qh$ is skew-hermitian, whence it induces a skew-symmetric form in the space $\Lambda_{1,Q}$.

\subparagraph{Forms and lattices at split places.}
Even though we are assuming throughout that $\psi$ does not split at $P$ (since otherwise we would not have a tree to study), we need to analyze the behaviour of these forms at split places for use in next section.
  Choose a point $P'\neq P$. We assume that 
  $\psi^{-1}(P')=\{Q_1,Q_2\}$. In this case, we recall that
  $L_{P'}:=K_{P'}\otimes_KL\cong L_{Q_1}\times L_{Q_2}\cong K_{P'}\times K_{P'}$. In this ring, the involution computes as $\overline{(a,b)}=(b,a)$, so in particular $\bar\rho=1-\rho$, if 
  $\rho$ denotes the idempotent $(1,0)$. In this case, the closure of the three dimensional space $L^3$ is isomorphic to $K_{P'}^3\times K_{P'}^3$ as a module over the ring 
  $L_{P'}$. Furthermore, note that $h(\rho v,\rho w)=\rho\bar\rho h(v,w)=0$, and the same holds for 
  $\bar\rho =1-\rho$. We conclude that 
  \begin{equation}\label{eqn hsplit}
  h\Big((v_1,v_2),(w_1,w_2)\Big)=b(v_1,w_2)+b(v_2,w_1),     
  \end{equation}
for a suitable symmetric bilinear form over $K$. Furthermore, the ``ring of integers'' is $B_{P'}\cong B_{Q_1}\times B_{Q_2}$, and lattices have the form $M=M_1\times M_2$. In this context,
unimodular $B$-lattices give rise to pairs of dual lattices,
with respect to the symmetric form $b$.

\section{On the number of cusps}\label{sec number}

In this section we parametrize cusps of $\overline{X}=G(A)\backslash X$, i.e.~the ``legs of the spider'', in terms of the Picard group of $B$. More precisely, we have the following result.

\begin{Prop}\label{prop cusps}
There are natural bijections between the following sets:
\begin{enumerate}
    \item the set $\cf(\overline{X})$ of cusps of the quotient graph $\overline{X}$ that are represented by $\cf(u,v)$ for some $(u,v)\in H(L,K)$;
    \item the double quotient $G(A)\backslash G(K)/\cal B(K)$;
    \item the set of $G(A)$-orbits of isotropic lines in $L^3$, where the action comes from the natural action of $\SL_{3}(L)$;
    \item the Picard group $\pic(B)$.
\end{enumerate}
\end{Prop}

\begin{proof}
Let us prove that $\cf(\overline{X})$ is in bijection with $G(A)\backslash G(K)/\cal B(K)$. Consider the set of rays in $X$
\[\Omega:=\lbrace w\cdot \mathfrak{r}(u,v):  (u,v) \in H(L,K) , w \in \lbrace 1,\mathrm{s}  \rbrace\rbrace,\]
and define $\partial_{\infty}(\Omega):=\lbrace  \partial_{\infty}(\rf): \rf \in \Omega  \rbrace$. We claim that $\partial_{\infty}(\Omega)$ is stable under the natural action of $G(K)$ on $\partial_\infty(X)$. Indeed, by definition of $\rf(u,v)$, we know that $\partial_{\infty}(\Omega)=\mathcal{U}_{a}(K)\{1,\mathrm{s}\}\cdot \partial_{\infty}(\rf(\infty))$. Moreover, the stabilizer of $\partial_\infty(\rf(\infty))$ in $G(K)$ is $\mathcal{B}(K)$. Then, by Bruhat decomposition we have
\begin{align*}
G(K)\cdot \partial_{\infty}(\rf(\infty)) &=(\mathcal{B}(K) \cup \mathcal{U}_{a}(K) \mathrm{s} \mathcal{B}(K))\cdot \partial_{\infty}(\rf(\infty))\\
&= \partial_{\infty}(\mathfrak{r}(\infty))\cup \mathcal{U}_{a}(K) \mathrm{s}\cdot \partial_{\infty}(\rf(\infty))\\
&=\mathcal{U}_{a}(K)\cdot \partial_{\infty}(\mathfrak{r}(\infty)) \cup \mathcal{U}_{a}(K) \mathrm{s}\cdot\partial_{\infty}(\rf(\infty))\\
&= \partial_{\infty}(\Omega),
\end{align*}
so that $\partial_\infty(\Omega)$ is just the $G(K)$-orbit of $\partial_\infty(\rf(\infty))$.
In particular, we see that there is a natural bijection between $\partial_{\infty}(\Omega)$ and $G(K)/\cal B(K)$ and hence $G(A)\backslash\partial_\infty(\Omega)\simeq G(A)\backslash G(K)/\cal B(K)$.

On the other hand, Proposition \ref{prop link} implies that $\cf(\overline{X})$ is in bijection with the quotient $G(A)\backslash\partial_\infty(\Omega)$. This gives the desired bijection.\\

Let us prove now that $G(A)\backslash G(K)/\cal B(K)$ is in bijection with the set of $G(A)$-orbits of isotropic lines in $L^3$ with respect to the natural action of $G(A)\subset\SL_{3}(L)$ on $L^3$. Since every isotropic line
in $L^3$ is contained in a hyperbolic plane, which has an isometry 
switching its two isotropic lines, Witt's theorem indeed shows that $G(K)$ 
acts transitively on the set of such lines. Moreover, the stabilizer of the line generated by $(1,0,0)$, which is isotropic, is easily seen to correspond to upper triangular matrices, i.e.~$\cal B(K)$. This implies that isotropic lines in $L^3$ are in natural correspondence with the set $G(K)/\cal B(K)$, on which $G(A)$ acts on the left, whence the second bijection.\\

Finally, let us relate isotropic lines in $L^3$ with the Picard group of $B$. In order to do this, we need the following lemma, which follows from \cite[IX, Thm.~4.15]{Berhuy}.

\begin{Lem}\label{lem sizing}
Lel $I_1,\dots,I_n,J_1,\dots,J_n\subseteq L$ be $B$-fractional ideals. If
\begin{equation}\label{eqn iso22}
I_1\times\cdots\times I_n\cong J_1\times\cdots\times J_n
\end{equation}
 as $B$-modules, then $I_1\cdots I_n\cong J_1\cdots J_n$.
\end{Lem}

Denote by $\Lambda_0$ the canonical lattice $B^3\subset L^3$. Let $v$ be an arbitrary isotropic vector. Then $Lv\cap\Lambda_0=Iv$ for some fractional ideal $I$. Certainly choosing a different generator for the space $Lv$ would replace $I$ by another ideal in the same class, so the isotropic line completely determines the ideal class of $I$. This defines a  map from the set of isotropic lines to $\pic(B)$. So, in order to establish the last bijection, we only need to prove two facts:
\begin{enumerate}
    \item Two lines that fall into the same class $[I]\in\pic(B)$ are in the same $G(A)$-orbit, i.e.~there is an isometry of $L^3$ that maps one line to the other.
    \item For every fractional ideal $J$ there exists an isotropic vector $v_J$ satisfying $Lv_J\cap\Lambda_0=Jv_J$.
\end{enumerate}

Consider once again a fixed isotropic vector $v$ and its corresponding ideal $I$. Let $Q'\neq Q$ be a closed point in $D$, so that it defines a prime ideal in $B$ and hence a finite place in $L$ with valuation $\omega'$. Denote $\Lambda_{0,Q'}$ the local lattice $\Lambda_0\otimes_B\cal O_{Q'}\subset L_{Q'}^3$. Fix a uniformizer $\pi'$ for this place and let $v'={\pi'}^{\omega'(I)}v\in L_{Q'} v$ be a local generator of the one dimensional lattice $I_{Q'}v$. Then $L_{Q'} v\cap \Lambda_{0,Q'}=\mathcal{O}_{Q'}v'$. We claim that
$h(\Lambda_{0,Q'},v')=\mathcal{O}_{Q'}$.

Indeed, consider the cover $\psi:D\to C$ corresponding to $L/K$, and set $P'=\psi(Q')$. Assume $\psi$ is 
ramified at $P'$. Then, as we saw in \S\ref{sec lattices}, we have an induced non-singular bilinear form 
$h_{Q'}:\kappa_{Q'}^3\times \kappa_{Q'}^3\to\kappa_{Q'}$, whence 
$h_{Q'}(v',w')=1$, for some $w'\in \kappa_{Q'}^3$, since $v'\neq 0$. The same argument works if $\psi$ is inert at $P'$, except that now $h$ is hermitian. Finally, if $\psi$ is split at $P'$, i.e.~$\psi^{-1}(P')=\{Q_1,Q_2\}$ with $Q_1=Q'$, the same argument can be applied using Formula~\eqref{eqn hsplit}. The claim follows.\\

From the claim, applied to every $Q'\neq Q$, we conclude that $h(\Lambda_0,Iv)=B$, or equivalently $h(\bar{I}\Lambda_0,v)=B$, whence there exists an element 
$w\in\bar{I}\Lambda_0$ such that $h(w,v)=h(v,w)=1$. Furthermore, note that $h(w,w)\in\bar{I}I$,
whence $h(w,w)v\in\bar{I}(Iv)\subseteq\bar{I}\Lambda_0$. Thus, up to replacing $w$ by $w-\frac12h(w,w)v$, we may assume it is isotropic.
We conclude that $\bar{I}^{-1}w\subseteq\Lambda_0$ and 
$h(\bar{I}^{-1}w,Iv)=B$. In particular, $\Lambda_1:=\bar{I}^{-1}w\oplus Iv$
is a unimodular sublattice of $\Lambda_0$, whence we have a splitting 
$\Lambda_0=\Lambda_1\perp\Lambda_1^{\perp}$.

Now let $v'$ be another isotropic vector satisfying  $Lv'\cap\Lambda_0=Iv'$.
Then we can write $\Lambda_0=\Lambda'_1\perp(\Lambda'_1)^{\perp}$, where
$\Lambda'_1=\bar{I}^{-1}w'\oplus Iv'$ for a suitable isotropic vector $w'$.
It is immediate that $\Lambda_1$ and $\Lambda'_1$ are isometric. By Lemma \ref{lem sizing}
the lattices $\Lambda_1^{\perp}$ and $(\Lambda'_1)^{\perp}$ are isomorphic,
and they are also isometric by discriminant considerations.
More precisely, we can assume 
$\Lambda_1^{\perp}=I^{-1}\bar{I}z$ and $(\Lambda'_1)^{\perp}=I^{-1}\bar{I}z'$, where $h(z,z)=h(z',z')=1$.
With this we can define an isometry $\Lambda_0\to\Lambda_0$ and hence an isometry $L^3\to L^3$ (i.e.~an element in $G(A)$ seen as an element in $\mathrm{SL}_3(B)$) which takes $v$ to $v'$. This proves 1.

In order to prove 2, we observe that, for any ideal $J$, the lattice
\[(\bar{J}^{-1}w\oplus Jv)\perp J^{-1}\bar{J}z,\]
with $z$ as above, is a unimodular lattice that is locally isometric to $\Lambda_0$, whence it is also globally isometric since the group $G$ has strong approximation (cf.~\cite[Thm.~A]{PrasadSA}). The result follows.
\end{proof}

\section{Proof of the main result}\label{sec proof main}
We can now give a proof of Theorem \ref{principal result}. The assertion for finite $\bb F$ will be proved further below.

\begin{proof}[Proof of Theorem \ref{principal result} (minus the last assertion)]
By Corollary \ref{cor legs} and Proposition \ref{prop cusps}, for every element $\sigma\in\pic(B)$ we can define the ray $\cf(\sigma)$ as the union of the cusp rays $\cf(u,v)$ whose class in $\cf(\overline{X})$ corresponds to $\sigma$. Defining $v_\sigma$ to be the tip of $\cf(\sigma)$, Lemma \ref{lemma valency cusps} ensures then that we can write
\[\overline{X}=Y\cup\bigsqcup_{\sigma\in\pic(B)}\cf(\sigma);\]
with $v(Y)\cap v\big(\cf(\sigma)\big)=\{v_\sigma\}$ and $e(Y)\cap e\big(\cf(\sigma)\big)=\emptyset$ for some subgraph $Y\subset \overline X$.

We are only left to prove that $Y$ is connected. For this we choose two points $\bar{y}$ and $\bar{y}'$ in $Y$ and two preimages $y$ and $y'$ in the Bruhat-Tits tree $X$. Find a walk on $X$ from $y$ to $y'$, and look at its image in $\overline{X}$. If this walk contains a vertex in one of the rays $\cf(\sigma)$, then it can do so only by going through the vertex $v_\sigma$, since every vertex in $\cf(\sigma)$ has valency two by Lemma \ref{lemma valency cusps}. If we remove the section of the walk between the first and the last time it visits the vertex $v_\sigma$, we get a shorter walk, so we can iterate this process until we find a walk that contains only vertices in $Y$. This finishes the proof.
\end{proof}

We are now left with the last assertion of Theorem \ref{principal result}, which states that, when $\bb F$ is a finite field, the graph $Y$ corresponding to the ``body of the spider'' is actually \emph{finite} (this is not the case for general $\bb F$). In order to prove this, we introduce first some definitions and a result by Bux, K\"ohl and Witzel here below.\\

Let $\Delta$ be the spherical building of $G(K)$, i.e.~the simplicial complex that is the realization of the poset of proper $K$-parabolic subgroups of $G$. Any vertex $\xi$ of $\Delta$ corresponds to a maximal $K$-parabolic subgroup $\mathcal{P}_{\xi}$ of $G$. In particular, the building $\Delta$ is trivial if and only if $G$ is anisotropic over $K$. In our context, $\Delta$ has dimension $0$ since $G$ has dimension $1$. Indeed, $\Delta$ can be isometrically embedded in $\partial_{\infty}(X)$. In this sense, we have that every $\xi \in \Delta$ can be represented by the equivalence class of a ray $\rf(u,v)$, for some $(u,v) \in H(L,K)\cup\{\infty\}$ (c.f.~sections \S\ref{sec trees} and \S\ref{sec number}). In all that follows we identify $\Delta$ with its image in $\partial_{\infty}(X)$. Then, for any pair $(v,\xi)$ of a vertex $v \in v(X)$ and a vertex $\xi \in \Delta$ we can define a ray $\rf(v,\xi)$ as the geodesic ray in $X$ from $v$ and whose visual limit is $\xi$.

The following result is a direct application of Theorem \ref{thm witzel}, proved by Bux, K\"ohl and Witzel.

\begin{Thm}\label{teo BGW}
Assume that $\bb F$ is a finite field. Then there exist finitely many rays $ \lbrace \rf_i=\rf(v_i,\xi_i) \rbrace_{i=1}^s$ for some $\xi_i \in \Delta$ and $v_i \in X$, and a constant $L_0$ such that every point in $X$ is within distance $L_0$ to the orbit of some ray $\rf_i$.
\end{Thm}

With this, we immediately obtain the finiteness of $Y$ in Theorem \ref{principal result} as follows. 

\begin{proof}[End of proof of Theorem \ref{principal result}]
Let $Y$ and $(\cf(\sigma),v_\sigma)_{\sigma \in \mathrm{Pic}(B)}$ be given by Theorem \ref{principal result}. As it was stated above, since the rays $\rf_i$ given by Theorem \ref{teo BGW} are defined by $\xi_i\in\Delta$, they are equivalent to a ray of the form $\rf(u,v)$ for some $(u,v)\in H(L,K)\cup\{\infty\}$. This immediately implies by construction that the image of $\rf_i$ in $\overline{X}$ is equivalent to one of the rays $\cf(\sigma)$. Set $L:=L_0+ \max \lbrace m_i\rbrace_{i=1}^s$, where $m_i$ is the distance between the origins of $\rf_i$ and the corresponding $\cf(\sigma)$. Then clearly every vertex in $\overline{X}$ is at a distance at most $L$ from one of the cusp rays $\cf(\sigma)$. On the other hand, the finiteness of $\bb F$ ensures the finiteness of $\pic(B)$ and the finiteness of the valency of every vertex in $X$. Putting everything together, we get the finiteness of $Y$.
\end{proof}

\begin{Rq}
Note that Theorem \ref{principal result} allows us to give a more precise version of the results of Bux, K\"ohl and Witzel in the context of a finite field $\bb F$. Indeed, while their results give fundamental domains in $X$ for the action of $G(A)$, our result describes more precisely the structure of the quotient graph $\overline{X}$. For instance, knowing that the fundamental domain ``looks like a spider'' does not imply that this is the case for $\overline X$ since we do not know a priori the valency of the vertices in the ``legs'' after passing to the quotient. Moreover, we are able to describe precisely the number of legs. This is impossible (or irrelevant) by definition with Bux, K\"ohl and Witzel's approach since they start from the $K$-group and then fix an arbitrary model over $A$, while we start with an $A$-group from the very beginning.
\end{Rq}

\section{Applications}\label{sec app}

In this section we use the theory developed above in order to study further the structure of the group $G(A)$.

\subsection{Amalgams}
In this section we analyze the structure of $G(A)$ as an amalgam. In order to do this, we extensively use Bass-Serre theory (cf. \cite[Chap.~I, \S 5]{S}).

To start, we choose a maximal tree $T$ of $\overline{X}$ and a lift $j: T \rightarrow X$. Equivalently, let $T$ be the union of the cusp rays $\cf(\sigma)$ of $\overline{X}$ for $\sigma \in \pic(B)$, with a maximal tree of $Y$. Let $(s_X, t_X, r_X)$, $(s_{\overline{X}}, t_{\overline{X}}, r_{\overline{X}} )$ be the source, target and reverse maps of the respective graphs $X$ and $\overline{X}$. Let $O$ be an orientation of the edges in $\overline{X}$, and set $o(y)=0$, if $y \in O$, and $o(y)=1$, if $y \not\in O$, or equivalently if $r_{\overline{X}}(y) \in O$. Then, we claim that we can extend $j$ to a section $j: e(\overline{X}) \to e(X)$ such that $j(r_{\overline{X}}(y))=r_X(j(y))$, for all $y \in e(\overline{X})$. Indeed, it suffices to define $j(y)$, for $y \in O \smallsetminus e(T)$. In this case we choose $j(y)$ so that $s_X(j(y)) \in v(j(T))$. Then, we have $s_X(j(y))=j(s_{\overline{X}}(y))$, for all $y \in O$. We also choose $g_{y} \in G(A)$ satisfying $t_X(j(y))=g_{y}\cdot j(t_{\overline{X}}(y))$. This is possible since $t_X(j(y))$ and $j(t_{\overline{X}}(y))$ have the same image $t_{\overline{X}}(y)$ in $v(\overline{X})$. So, we extend the map $y \mapsto g_{y}$ to all edges in $\overline{X}$ by setting $g_{\bar{y}}=g_{y}^{-1}$, for all $y \in e(\overline{X})$, and $g_{y}=\id$, for all $y \in e(T)$. Thus, for each $y \in e(\overline{X})$ we get
\[s_X(j(y))= g_{y}^{-o(y)} j(s_{\overline{X}}(y)), \quad t_X(j(y))=g_{y}^{1-o(y)} j(t_{\overline{X}}(y)), \quad \forall y \in e(\overline{X}).\]

We need to define a graph of groups $(\mathfrak{g},\overline{X})=(\mathfrak{g}(T),\overline{X})$ associated to $\overline{X}$ and $T$ (cf.~\cite[Chap.~I, \S 4.4]{S}). This amounts to defining the following data
\begin{itemize}
\item For each vertex $\overline{v}\in v(\overline{X})$, we define the group $G_{\overline{v}}$ as the stabilizer in $G(A)$ of $j(\overline{v}) \in v(X)$;
\item for each edge $y\in e(\overline{X})$, we define the group $G_{y}$ as the stabilizer in $G(A)$ of $j(y) \in e(X)$;
\item for each pair $(\overline{v},y)$ where $\overline{v}$ is the target vertex of the edge $y$, we define a morphism $f_y: G_{y}\to G_{\overline{v}}$ by $g \mapsto g_{y}^{o(y)-1} g g_{y}^{1-o(y)}$. This definition is legitimate since we have $g_{y}^{o(y)-1} G_{j(y)} g_{y}^{1-o(y)} \subseteq G_{j(t_{\overline{X}}(y))}$.
\end{itemize}
We can define the fundamental group associated to this graph of groups. Indeed, let $F(\mathfrak{g}, \overline{X})$ be the group generated by the groups $G_{\overline{v}}$, where $\overline{v} \in v(\overline{X})$, and elements $a_y$ for each $y\in e(\overline{X})$, subject to the relations
\[a_{r_{\overline{X}}(y)}=a_y^{-1}, \quad \text{ and } \quad a_y f_y(b) a_y^{-1}= f_{r_{\overline{X}}(y)}(b), \quad \text{ for all } y \in e(\overline{X}) \text{ and } b \in G_{y}. \]
Then, the fundamental group $\pi_1(\mathfrak{g})=\pi_1(\mathfrak{g},\overline{X})$ of $(\mathfrak{g},\overline{X})$ is, by definition, the quotient of $F(\mathfrak{g},\overline{X})$ by the normal subgroup generated by the elements $a_y$ for $y\in e(T)$. Thus, if we denote by $h_y$ the image of $a_y$ in $\pi_1(\mathfrak{g}, \overline{X})$, the group $\pi_1(\mathfrak{g}, \overline{X})$ is generated by $G_{\overline{v}}$ for $\overline{v} \in v(\overline{X})$ and the elements $h_y$, for $y \in e(\overline{X})$, subject to the relations
\[\begin{array}{cl}
    h_{r_{\overline{X}}(y)}=h_y^{-1}, & \text{ for all } y \in e(\overline{X}),\\
    h_y  f_y(b) h_y^{-1}= f_{r_{\overline{X}}(y)}(b), & \text{ for all } y \in e(\overline{X}) \text{ and } b \in G_{y},\\
    h_y=\id, & \text{ for all } y \in e(T).
\end{array}\]
It can be proven that the group $\pi_1$ is independent, up to isomorphism, 
of the choice of the graph of groups $\mathfrak{g}$, and in particular 
of the tree $T\subset \overline{X}$.

One of the fundamental results from Bass-Serre Theory implies that $G(A)$ can be always characterized from the action of $G(A)$ on $X$ (cf. \cite[Chap.~I, \S 5.4]{S}). Indeed, it is isomorphic to the fundamental group $\pi_1(\mathfrak{g})=\pi_1(\mathfrak{g}, \overline{X})$. We use this fact with no further explanation in all that follows.\\

In the sequel, we assume that $\overline{X}$ is combinatorially finite. This will always be the case when $\bb F$ is finite thanks to Theorem \ref{principal result}. We want to prove that $G(A)$ is an amalgam of a group that depends on $Y$ (the ``body'' of the spider) and finitely many groups $\mathcal{G}_\sigma$, for $\sigma\in\pic(B)$ (the ``legs'' of the spider).

Let $\cf(\sigma)$ with $\sigma\in\pic(B)$ be a cusp of $\overline{X}$ and fix a $(u,v)\in H(L,K)$ such that $\cf(\sigma)=\cf(u,v)$. Let $\rf_0(u,v)$ be the subray of $\rf(u,v)$ whose tip vertex is $\lambda_{N(u,v)+1}(u,v)$. This is a lift of $\cf(\sigma)$ to $X$. Let $x_n=\lambda_n(u,v)\in\rf_0(u,v)$ be a vertex different from the tip, and let $y_n$ be the edge joining $x_n$ with $x_{n+1}$. It follows from Proposition \ref{prop stab}.1, that
\[\cdots \subset \mathrm{Stab}_{G(A)}(x_{n}) \subset \mathrm{Stab}_{G(A)}(x_{n+1}) \subset \mathrm{Stab}_{G(A)}(x_{n+2}) \subset \cdots.\]
In particular, we see that $\mathrm{Stab}_{G(A)}(y_{n})=\mathrm{Stab}_{G(A)}(x_{n})$. From this, we see that $\stab(u,v)$, which is by Proposition \ref{prop stab}.2 the union of all these stabilizers, is isomorphic to the fundamental group $\mathcal{G}_{\sigma}$ of the graph of groups associated to $\cf(u,v)$.

In all that follows, we denote by $(u,v)=(u,v)_{\sigma} \in H(L,K)$ the element fixed above for $\sigma \in \pic(B)$. Let $\mathcal{H}$ be the fundamental group of the restriction of the graph of groups $\mathfrak{g}$ to $Y$. Let $v_\sigma=\overline{\lambda_{N(u,v)+1}(u,v)}$ be the intersection of $\cf(\sigma)$ and $Y$ as in Theorem \ref{principal result} and let $G_{v_{\sigma}}$ be the corresponding group in the graph of groups, that is $G_{v_{\sigma}}=\stab(u,v,N(u,v)+1)$. Then there are canonical injections $G_{v_\sigma} \to \mathcal{H}$ and $G_{v_\sigma} \to \mathcal{G}_{\sigma}$.

\begin{Thm}\label{thm amalgam}
Assume that $\overline{X}$ is combinatorially finite. Then, $G(A)$ is isomorphic to the sum of $\mathcal{G}_\sigma$, for $\sigma \in \pic(B)$, and $\mathcal{H}$, amalgamated along the groups $G_{v_\sigma}$, according to the previously defined injections. Moreover, $\mathcal{G}_{\sigma}$ is an extension of a subgroup of $\bb F^*$ by $H(u,v)$, where $(u,v)=(u,v)_{\sigma}$ and $H(u,v)$ is defined in Lemma \ref{Lemma M}.

In particular, when $\mathbb{F}$ is finite field, we have that $\mathcal{H}$ is finitely generated and $G_{v_\sigma}$ is finite for every $\sigma\in\pic(B)$.
\end{Thm}

Recall that in Lemma \ref{Lemma M} we proved that $H(u,v)$ is a subgroup of $H(L,K)$ whose intersection with $H(L,K)^0\cong K$ is a finitely generated $A$-module $H(u,v)^0$ and such that the quotient $H(u,v)/H(u,v)^0$ is also a finitely generated $A$-submodule of $H(L,K)/H(L,K)^0\cong L$. This fact and Theorem \ref{thm amalgam} imply then Theorem \ref{main thm amalgam}.

\begin{proof} As Serre points out in \cite[Chap.~II, \S 2.5, Th. 10]{S}, if we have a graph of groups $\mathfrak{h}$, which is obtained by ``gluing'' two graphs of groups $\mathfrak{h}_1$ and $\mathfrak{h}_2$ by a tree of groups $\mathfrak{h}_{12}$, then there exist two injections $\iota_1:\mathfrak{h}_{12} \to \mathfrak{h}_{1}$ and $\iota_2: \mathfrak{h}_{12} \to \mathfrak{h}_{2}$, such that $\pi_1(\mathfrak{h})$ is isomorphic to the sum of $\pi_1(\mathfrak{h}_1)$ and $\pi_1(\mathfrak{h}_2)$, amalgamated along $\pi_1(\mathfrak{h}_{12})$ according to $\iota_1$ and $\iota_2$. Given that $\cal G_\sigma=\stab(u,v)$, the first part of the theorem follows from Theorem \ref{principal result} and Lemma \ref{Lemma M}.

Assume now that $\mathbb{F}$ is a finite field. Then, each vertex stabilizer is finite, since it is the intersection of a compact subgroup of $G(K_P)$ with a discrete one. In particular, we obtain that $G_{v_{\sigma}}$ is finite. Moreover, in this context $\mathcal{H}$ is the fundamental group of a graph of groups whose underlying graph is finite and whose vertex stabilizers are also finite. So, we conclude that $\mathcal{H}$ is finitely generated. The result follows.
\end{proof}

% we have that $G(A)$ is isomorphic to the sum of the groups $G_{\overline{v}}$ for $\overline{v} \in v(\overline{X})$, and the fundamental group $\pi_1(\overline{X})$ amalgamated along some common subgroups. In all that follows we analyze the groups $G_{\overline{v}}$ in order to exploit this isomorphism and its cohomological consequences.\\

\begin{Cor}\label{cor G(A) not fg}
The group $G(A)$ is not finitely generated
\end{Cor}

\begin{proof} Fix an element $\sigma \in \pic(B)$. Let $\mathcal{G}_{\sigma}'$ be the group obtained by summing $\mathcal{H}$ with the groups $\mathcal{G}_{\sigma'}$ for $\sigma' \neq \sigma$. Then, by Theorem \ref{thm amalgam} we have that $G(A)$ is isomorphic to the sum of $\mathcal{G}_{\sigma}'$ with $\mathcal{G}_{\sigma}$, amalgamated along $G_{v_{\sigma}}$. Recall that $\mathcal{G}_{\sigma}$ is the union of the strictly increasing sequence of groups $\stab(u,v,n)$, for $n > N(u,v)$, where $(u,v)=(u,v)_{\sigma}$. In particular, all these groups contain $G_{v_{\sigma}}= \stab(u,v,N(u,v)+1)$. So, for each $n > N(u,v)$, we define $G(A)_n$ as the sum of $\mathcal{G}_{\sigma}'$ with $\stab(u,v,n)$, amalgamated along $G_{v_{\sigma}}$. Then, we get that $G(A)$ is the union of the strictly increasing sequence $G(A)_n$, which implies that it is not finitely generated. 
\end{proof}

\subsection{Homology}
We follow \cite[II.2.8]{S}, where Serre studies the homology of the group $G(A)$ when $G=\SL_2$ and $\bb F$ is a finite field. We keep notations as above and \textbf{we assume that $\bb F$ is a finite field}. In particular, $\bar X$ is combinatorially finite.

\begin{Prop}\label{prop exact sequence}
Let $M$ be a $G(A)$-module. Then there is a long exact sequence of homology groups
\[
\xymatrix{  \cdots\ar[r]& H_{i+1}(G(A),M)\ar[r]& 
\displaystyle{\bigoplus_{\sigma \in \pic(B)}} H_{i}(G_{v_\sigma},M)\ar[r]&
H_{i}(\mathcal{H},M) \oplus \displaystyle{\bigoplus_{\sigma \in \pic(B)}} H_{i}(\mathcal{G}_\sigma,M)
\ar`r[d]`[lll]`[llld]`[llldr] [dll]  \\
{} & H_{i}(G(A),M)\ar[r]& \cdots,& }
\]
and a long exact sequence of cohomology groups
\[
\xymatrix{  \cdots\ar[r]& H^{i}(G(A),M) \ar[r]& 
H^{i}(\mathcal{H},M) \oplus \displaystyle{\bigoplus_{\sigma \in \pic(B)}} H^{i}(\mathcal{G}_\sigma,M)\ar[r]&
\displaystyle{\bigoplus_{\sigma \in \pic(B)}} H^{i}(G_{v_\sigma},M)
\ar`r[d]`[lll]`[llld]`[llldr] [dll]  \\
{} & H^{i+1}(G(A),M)\ar[r]& \cdots.& }
\]

\begin{comment}
USE THESE EQUATIONS IF THE ONES ABOVE DO NOT FIT

\begin{multline*}
\cdots \to H_{i+1}(G(A),M) \to \bigoplus_{\sigma \in \pic(B)} H_{i}(G_{v_\sigma},M)\to\\
\to H_{i}(\mathcal{H},M) \oplus \bigoplus_{\sigma \in \pic(B)} H_{i}(\mathcal{G}_\sigma,M)
\to H_{i}(G(A),M)\to \cdots,
\end{multline*}
and a long exact sequence of cohomology groups
\begin{multline*}
\cdots \to H^{i}(G(A),M) \to H^{i}(\mathcal{H},M) \oplus \bigoplus_{\sigma \in \pic(B)} H^{i}(\mathcal{G}_\sigma,M)\to\\
\to \bigoplus_{\sigma \in \pic(B)} H^{i}(G_{v_\sigma},M)
\to H^{i+1}(G(A),M)\to \cdots,
\end{multline*}
\end{comment}

\end{Prop}

\begin{proof}
By \cite[II.2.8, Prop.~13]{S}, it suffices to define an action of $G(A)$ on a suitable tree $T_0$ satisfying the following properties:
\begin{itemize}
    \item There exists a system of representatives $(w_\sigma)_{\sigma\in\pic(B)\cup\{\ast\}}$ of $v( T_0)$ such that the stabilizer of $w_\sigma$ in $G(A)$ is $\cal G_\sigma$, for $\sigma\in\pic(B)$,
    and the stabilizer of $w_\ast$ is $\cal H$.
    \item There exists a system of representatives $(f_\sigma)_{\sigma\in\pic(B)}$ of $e(T_0)$ such that the stabilizer of $f_\sigma$ in $G(A)$ is $G_{v_\sigma}$, for any $\sigma\in\pic(B)$.
    \end{itemize}
In order to define $T_0$, we consider first the tree of groups naturally associated to the amalgam obtained in Theorem \ref{thm amalgam} above, namely (cf.~\cite[I.4.4]{S}):
\[  
\xygraph{
!{<0cm,0cm>;<.8cm,0cm>:<0cm,.8cm>::}
!{(0,1)}*+{\bullet}="P" !{(-0.3,1)}*+{{}^{\cal H}}="Pn"
!{(2,1.6)}*+{}="B1" !{(2,0.4)}*+{}="B2" 
!{(3,2)}*+{\bullet}="A2"  !{(3.5,2)}*+{{}^{\cal G_{\sigma}}}="A2n" !{(1.4,1.8)}*+{{}^{ G_{v_\sigma}}}="A2m"
!{(3,0)}*+{\bullet}="A3"!{(3.5,0)}*+{{}^{\cal G_{\omega}}}="A3n"!{(1.4,0.3)}*+{{}^{ G_{v_\omega}}}="A3m"
 "P"-"A2" "P"-"A3"  "A2n"-@{.}@/_-.5pc/"A3n" 
%"A2"-@{.}@/_-.5pc/"A3" 
"B1"-@{.}@/_-.5pc/"B2"  }
\]
Then \cite[I.4.5, Thm.~9]{S} gives us precisely the desired tree $T_0$. Alternatively, one could follow directly the second proof of \cite[II.2.5, Thm.~10]{S}.
\end{proof}

Still following Serre, we can use Theorem \ref{thm amalgam} in order to get the following result.

\begin{Prop}\label{prop homo}
Let $M$ be a $G(A)$-module that is finitely generated as an abelian group. Let $p=\mathrm{char}(\mathbb{F})$. Then:
\begin{itemize}
    \item[(a)] For $i \geq 2$, the morphism 
    \[\phi: H_{i}(\mathcal{H},M) \oplus \bigoplus_{\sigma \in \pic(B)} H_{i}(\mathcal{G}_\sigma,M) \to H_{i}(G(A),M),\]
    has finite kernel and cokernel. For $i=1$, $\ker(\phi)$ is finite and $\mathrm{coker}(\phi)$ is finitely generated.
    \item[(b)] For $i \geq 2$, the group $H_i(\mathcal{H},M)$ is finite. For $i \leq 1$ it is finitely generated.
    \item[(c)] For $i \geq 1$ and $\sigma \in \pic(B)$, the group $H_i(\mathcal{G}_\sigma,M)$ is the direct sum of a finite group and a countable $p$-primary torsion group.
\end{itemize}
\end{Prop}

\begin{proof}
By \ref{prop exact sequence} and the finiteness of $\pic(B)$, in order to prove assertion (a), it
suffices to prove that the groups $H_i(G_{v_\sigma},M)$ are finite, for $i\geq 1$, and that $H_0(G_{v_\sigma},M)$ is finitely generated. Now these assertions follow immediately from the finiteness of $G_{v_\sigma}$ (proved in Theorem \ref{thm amalgam}) and the hypothesis on $M$.\\

Let us prove (b). We claim that $\cal H$ is virtually free of finite rank, which implies the assertion. Now this is a direct consequence of Theorem \ref{thm amalgam} and \cite[II.2.6, Prop.~11]{S} applied to the graph of groups $(\cal H,Y)$.\\

Finally, for (c), fix $\sigma\in\pic(B)$. Lemma \ref{Lemma M}.3, Proposition \ref{prop stab}.2 and the finiteness of $\bb F$ tell us that $\cal G_\sigma$ fits into an exact sequence
\[1\to U_\sigma \to \cal G_\sigma \to F_\sigma \to 1,\]
where $F_\sigma$ is a finite abelian group of order prime to $p$ and $U_\sigma$ is a direct limit of finite $p$-groups, hence a countable $p$-primary torsion group. Consider then the Hochschild-Serre spectral sequence
\[H_i(F_\sigma,H_j(U_\sigma,M))\Rightarrow H_{i+j}(\cal G_\sigma,M).\]
Since $U_\sigma$ is a $p$-primary torsion group and $F_\sigma$ has order prime to $p$, we have $H_i(F_\sigma,H_j(U_\sigma,M))=0$ whenever $i,j\neq 0$. Moreover, the term $H_i(F_\sigma,H_0(U_\sigma,M))$ is clearly finite of order prime to $p$ and the term $H_0(F_\sigma,H_j(U_\sigma,M))$ is a quotient of $H_j(U_\sigma,M)$, which is clearly a countable $p$-primary torsion group. This implies assertion (c).
\end{proof}

This last result allows us to understand the homology group $H_i(G(A),M)$ by using only the collection of groups $H_i(\mathcal{G}_\sigma,M)$ for $\sigma \in \pic(B)$. And these can be made explicit in theory thanks to the results in \S\ref{sec stab}. Indeed, the following corollary is an immediate consequence of Proposition \ref{prop homo}.

\begin{Cor}\label{cor G torsion}
The morphism
\[\phi_0: \bigoplus_{\sigma \in \pic(B)} H_{i}(\mathcal{G}_\sigma,M) \to H_{i}(G(A),M).\]
has finite kernel and cokernel for $i\geq 2$ and, for $i=1$, $\mathrm{ker}(\phi_0)$ is finite and $\mathrm{coker}(\phi_0)$ is finitely generated.

In particular, the group $H_i(G(A),M)$ is the direct sum of a finite group and a countable $p$-primary torsion group for $i \geq 2$.
\end{Cor}

The following result is also immediate.

\begin{Cor}
Assume that $M$ is a finite module whose order is prime to $p$. Then $H_i(G(A),M)$ is finite for every $i\geq 0$.
\end{Cor}

Finally, by taking $M=\Z$, we get the following results

\begin{Cor}
The abelian group $G(A)^{\mathrm{ab}}$ is finitely generated and its rank is bounded by the order of $\pic(B)$.
\end{Cor}

\begin{proof}
Indeed, since $G(A)^{\mathrm{ab}}=H_1(G(A),\Z)$, we see from the exact sequence in Proposition \ref{prop exact sequence} and Proposition \ref{prop homo} that the rank of $G(A)^{\mathrm{ab}}$ is equal to the rank of its image in
\[\bigoplus_{\sigma \in \pic(B)} H_{0}(G_{v_\sigma},\Z)=\Z^{\pic(B)}.\]
\end{proof}

\begin{Cor}
One has $H_i(G(A), \mathbb{Q})=0$, for $i \geq 2$, and $H_1(G(A),\mathbb{Q})$ is a $\mathbb{Q}$-vector space of finite dimension that is isomorphic to $H_1(\mathcal{H}, \mathbb{Q})$.
\end{Cor}

\begin{proof}
Remember that $H_i(G(A), \mathbb{Q})=H_i(G(A), \mathbb{Z}) \otimes \mathbb{Q}$. It follows from Corollary \ref{cor G torsion} that $H_i(G(A), \mathbb{Z})$ is a torsion group for $i \geq 2$, and then $H_i(G(A), \mathbb{Q})=0$. The same argument using Proposition \ref{prop homo} proves that $H_i(G_{v_\sigma}, \mathbb{Q})=0$ and $H_i(\cal G_{\sigma}, \mathbb{Q})=0$ for $i\geq 1$. Then the exact sequence in Proposition \ref{prop exact sequence} shows that the homomorphism $H_1(\mathcal{H}, \mathbb{Q}) \to H_1(G(A), \mathbb{Q})$ is an isomorphism.
\end{proof}

\section{An example: edges of \texorpdfstring{$\lambda_0(\infty)$}{the standard sector} and their image in \texorpdfstring{$\overline{X}$}{the quotient of the building}}\label{section lambda 0 infty}
Using the notations from \S\ref{sec stab}, let us consider the special vertex $\lambda_0(\infty)\in X$, which corresponds to the unimodular lattice $\Lambda_0=\mathcal{O}_Q^3$. The corresponding stabilizer is the set $G(C)=G(A)\cap G(\mathcal{O}_P)$. We can study several cases:

\paragraph{The isotrivial case.}
This case is well-documented in current literature \cite{Margaux}. This is the case where the group scheme $G$ is actually defined over $\mathbb{F}$ and thus $G(C)$ corresponds to its $\bb F$-points in the usual sense. In this case, when $Q$ is a point of degree 1, then the stabilizer $G(C)$ of $\lambda_0(\infty)$ is the unitary group of a hermitian form for a quadratic extension over $\mathbb{F}$. In this case, it is easy to show that this group acts transitively on the isotropic lines and thus the stabilizer of $\lambda_0(\infty)$ acts transitively on the set of edges coming out from it. In particular, the image of this vertex on the quotient graph will have valency 1.

\paragraph{The ramified case.}
Back to our setting, consider the case where $Q$ is a ramified point of degree 1. Then the hermitian form $h$ becomes a bilinear form over the residue field $\bb F$, as the action of the Galois group is trivial there. Then $G(C)$ becomes the orthogonal group corresponding to this bilinear form, and the transitivity of the stabilizer of $\lambda_0(\infty)$ on the set of edges coming out from it follows as before. So, once again, the image of this vertex on the quotient graph will have valency 1.

\paragraph{The unramified case.}
Consider now the unramified case, i.e.~where $P$ has degree 1 and $Q$ has degree 2. In this case we do have a hermitian form defined over the residue field, which is a quadratic extension of $\bb F$, but the group $G(C)$ does not correspond to the special unitary group of this form. Indeed, one can check directly that, in this case, $G(C)$ corresponds to the orthogonal group of the bilinear form over $\bb F$ with the same Gram matrix, which is a subgroup of the former. Thus, this case is much more involved, as it requires us to study orbits of isotropic lines in a hermitian space under a certain orthogonal subgroup of the unitary group. We present this computation below, for some particular base fields, to illustrate the difficulties of the general case.\\

We assume thus that $P$ has degree 1 and $Q$ has degree 2. In particular, $\kappa_P=\bb F$ and $\kappa_Q$  is a quadratic extension $\bb E=\mathbb{F}(\sqrt{a})$ of $\bb F$ whose Galois group is generated by $\sigma$. We identify $\mathbb{F}^3$ as a subset of $\bb E^3$, and write vectors in the latter as $z=w+\sqrt{a}v$, where both $w$ and $v$ belong to the former. Note that the spanned line $\langle z\rangle_{\bb E}$ has a generator in $\mathbb{F}^3$ precisely when $\langle z\rangle_{\bb E}=\langle \sigma(z)\rangle_{\bb E}$, or equivalently, when $w$ and $v$ are $\bb F$-linearly dependent.
Moreover, the $\bb E$-span
$\langle w,v\rangle_{\bb E}=
\langle z,\sigma(z)\rangle_{\bb E}$ is completely determined by the line and comes from a quadratic subspace of $\bb F^3$ by base change. This naturally splits the problem in two parts:
\begin{enumerate}
\item Classify $G(C)$-orbits of binary quadratic subspaces of $\mathbb{F}^3$ whose $\bb E$-span contains an isotropic line.
\item Find the orbits of isotropic $\bb E$-lines corresponding to each quadratic subspace.
\end{enumerate}
After this analysis is done, one should not forget to add the obvious orbit corresponding to the isotropic $\bb E$-lines that come from isotropic $\bb F$-lines, i.e., when $w$ and $v$ are $\bb F$-linearly dependent.\\

Let us deal with the first part. The formula 
$(x+y\sqrt a)(w+\sqrt{a}v)=
(xw+ayv)+\sqrt{a}(yw+xv)$
tells us that we can replace $w$ by any nontrivial vector in the space $\langle w,v\rangle_\mathbb{F}$. In particular, we can assume it to be isotropic unless $\langle w,v\rangle_\mathbb{F}$ is anisotropic as a quadratic space.
Furthermore, if $w$ is isotropic, then so is $v$, as $z$ is assumed to be isotropic. This implies that the span $\langle w,v\rangle_\mathbb{F}$ contains a basis of isotropic vectors, and it is either a totally isotropic subspace or a hyperbolic plane. The former is not possible as the full hermitian space is regular and three dimensional, whence its Witt index is $1$. We conclude that all binary spaces involved are regular, so Witt's Theorem tells us that the answer to the first part of the problem is the number of isometry classes of such subspaces. We can note then the following:

\begin{itemize}
\item We can always find a hyperbolic plane since the matrix of both the hermitian form and the bilinear form is simply
\[\begin{pmatrix} 0 & 0 & 1 \\ 0 & 1 & 0 \\ 1 & 0 & 0
\end{pmatrix}.\]
Thus the first and last vectors of this basis give a hyperbolic plane.

\item On the other hand, whether we find anisotropic subspaces, and how many, will depend on the field $\bb F$.
It is easy to prove that any norm form, where the space is generated by a vector of length $1$ and an orthogonal vector of length $-b$ is always represented, since the  hyperbolic plane is universal. Rescalings of it are more involved.
When $\mathbb{F}$ is finite, there is a unique isometry class of anisotropic subspaces, and it is a norm form, so there will be precisely one orbit. For the field of real numbers, the negative definite binary space is not represented by our form, while the positive definite space is, so there is also only one orbit. On the other hand, when $\mathbb{F}=\mathbb{Q}_2$ and $a=5$, the norm form of either $\mathbb{F}(\sqrt3)$ or $\mathbb{F}(\sqrt{7})$ corresponds to a suitable binary subspace.
\end{itemize}

Let us deal now with the second part of the problem, which we have to split in cases depending on the quadratic subspace we found.

\begin{itemize}
\item For the hyperbolic plane we can write $v=ru$, where $h(w,u)=1$, so that the isotropic line is the $\bb E$-span of $w+r\sqrt{a}u$. The orthogonal group is generated by the reflection that interchanges $w$ and $u$ and the maps sending $(w,u)$ to $(sw,s^{-1}u)$. The latter sends  $w+r\sqrt{a}u$ to $sw+rs^{-1}\sqrt{a}u=s(w+rs^{-2}\sqrt{a}u)$, so $r$ can be modified by a square. On the other hand, the reflection sends $w+r\sqrt{a}u$ to $u+r\sqrt{a}w=r\sqrt{a}(w+r^{-1}a^{-1}\sqrt{a}u)$. We conclude that these lines are in correspondence with $\mathbb{F}^*/\langle a,\mathbb{F}^{*2} \rangle$. In particular, when $\bb F$ is a finite field or the field of real numbers, there is a unique orbit.

\item In the case of anisotropic subspaces, since the situation depends on the field $\bb F$, we will only study further, here below, the particular case of norm forms. Note however that any anisotropic binary quadratic form is a rescaling of a norm form, so this is still a quite general case.

Moreover, as we noted above, norm forms are enough for a full analysis in the case of finite fields. Note also that in the real case the situation is even simpler: the only anisotropic space we got defines an anisotropic hermitian form, so it gives no isotropic $\bb E$-lines.
\end{itemize}

\paragraph*{Norm forms.}
For any quadratic extension $\mathbb{K}/\mathbb{F}$, we consider the norm form $N:\mathbb{K}\rightarrow\mathbb{F}$. The associated $\mathbb{E}$-hermitian form is the norm form of a quaternion algebra. In fact, if $\mathbb{H}=\mathbb{K}\oplus j\mathbb{K}$ is the quaternion algebra satisfying $jc=\bar{c}j$, for any $c\in\mathbb{K}$, and if $j^2=a$ is as above, then the norm of $w+jv$ is $N(w)-aN(v)$. This algebra has norm-zero elements  precisely when it is isomorphic to a matrix algebra. If $b$ is any such element, another quaternion spanning the same isotropic line has the form $cb$, where $c\in\mathbb{F}[j]^*\cong\mathbb{E}^*$. The special unitary group of the norm form on a quadratic extension consist of maps of the form $u\mapsto ux$  where $x\in \mathbb{K}$ has norm one. We conclude that an orbit of isotropic lines of the hermitian form corresponds to one or two elements in the set of double cosets $\mathbb{F}[j]^*\backslash \Phi/\mathbb{K}^1$, where $\Phi$ is the set of non-trivial elements of norm $0$ and $\bb K^1$ denotes the set of elements of norm $1$ in $\bb K$.

A simple way of characterizing these orbits is to write quaternions in the form $r+si$, where $i\in \mathbb{K}$ satisfies $i^2=b\in\mathbb{F}$. The isotropic line spanned by such a quaternion is completely determined by the quotient $x=r/s$. Post-multiplication by an element of the form $c+di\in \mathbb{K}$ replaces the invariant $x$ of the line $\mathbb{F}[j](r+si)$ by
the invariant $x'=\frac{cx+ad}{c+xd}$ of $\mathbb{F}[j](r+si)(c+di)$.\\

With all this, we can complete our search for the orbits of $G(C)$, and hence the valency of $\lambda_0(\infty)$ in the quotient graph $\overline{X}$, in the case where $\bb F$ is a finite field.

\begin{Lem}
In the above notations, if $\mathbb{F}$ is finite and $\mathbb{K}\cong\mathbb{L}$ is its unique quadratic extension, then the quotient set
$\mathbb{F}[j]^*\backslash \Phi/\mathbb{K}^1=\mathbb{F}[j]^*\backslash \Phi/\mathbb{K}^1\mathbb{F}^*$ has precisely two elements, while $\mathbb{F}[j]^*\backslash \Phi/\mathbb{K}^*$ is a singleton.
\end{Lem}

\begin{proof}
Since $\mathbb{H}$ is isomorphic to a matrix algebra, we just need to count the invertible elements of a matrix algebra. We have $p^2-1$ non-trivial matrices with a trivial first column, and $(p^2-1)p$ elements with a non-trivial first column $y$ and a second column in the span $\langle y \rangle$. This totals $p^3+p^2-p-1=(p+1)(p^2-1)$ invertible elements. On the other hand, the multiplicative group $\mathbb{F}[j]^*\times\mathbb{K}^1$ has the same order, with the element $(-1,-1)$
acting trivially. It suffices to prove that no other element has fixed points.
Set $cbu=b$, where $c=x+jy$ with $x,y\in\mathbb{F}$, and $u\in\mathbb{K}^*$. Set also $b=w+jv$. Expanding this product we get
$$(x+jy)(w+jv)u=(xw+ayv)u+j(yw+xv)u=w+jv.$$
Since $j$ and $1$ are linearly independent under right multiplication by $\mathbb{K}$,
we obtain $u^{-1}=x+ay(v/w)=y(w/v)+x$,
whence either $y=0$ or $(w/v)^2=a$. The latter condition implies $N(w)/N(v)=N(\sqrt{a})=-a$, but the condition that $b$ has null norm requires this quotient to be $a$ instead. We conclude that $y=0$, whence $xu=1$. The proof of the last statement is analogous, except that
the set of elements in $\mathbb{F}[j]^*\times \mathbb{K}^*$ acting trivially is $\{(r,r^{-1})|r\in\mathbb{F}^*\}$.
\end{proof}

\begin{Prop}
In the preceding notations, the valency of the vertex $\lambda_0(\infty)$ is $3$ whenever $\mathbb{F}$ is finite.
\end{Prop}

\begin{proof}
It suffices to determine whether a reflection on $\mathbb{K}$ interchanges the two orbits or not. Note that an element of the form $r+si$ spans an isotropic line if and only if $N(r/s)=a$. Furthermore, one particular rotation in $\mathbb{K}$ is the  linear map fixing $1$ and taking $i$ to $-i$. At the quaternionic level this map sends $r+si$ to $r-si$, whence replacing the invariant $r/s$ of the isotropic line by $-r/s$. It suffices therefore to prove that if there are elements $c$ and $d$ in $\mathbb{F}$
 satisfying $\frac{rc+asd}{sc+rd}=-\frac rs$, then the norm of $c+di$ must be a perfect square. Note that the preceeding identity is equivalent to $\frac cd=-\frac{r^2+as^2}{2rs}$,
 which in turns gives
 $$c^2-ad^2=d^2\left[\left(\frac cd\right)^2-a\right]=d^2\left[\left(\frac{r^2+as^2}{2sr}\right)^2-a\right]=\frac{d^2}{(2sr)^2}\left(r^2-as^2\right)^2.$$
 If the left hand side of the preceding chain of equalities is a perfect square, then the following element must belong to $\mathbb{F}$:
 $$\frac{r^2-as^2}{sr}=\frac rs-a\frac sr=\frac rs-\overline{\left(\frac rs\right)},$$
 but this is imposible, unless $r/s\in\mathbb{F}$. This finishes the proof.
 \end{proof}

\section{More examples: explicit computation of \texorpdfstring{$\overline{X}$}{the quotient of the building}}\label{examples}

\subsection{An example where \texorpdfstring{$L/K$}{L/K} ramifies at \texorpdfstring{$P$}{P}}

Let $C=\mathbb{P}_{\mathbb{F}}^1$, $P$ be the point at infinity of $C$ and $L=\mathbb{F}(\sqrt{t})$. Then $\psi$ corresponds to the (unique) 2-cover $\bb P^1_{\bb F}\to \bb P^1_{\bb F}$. In particular, $A=\mathbb{F}[t]$, $B=\mathbb{F}[\sqrt{t}]$, $g_C=g_D=0$, $d=\deg(P)=1$ and $e_P=2$. Since we are in the ramified case, the analysis done in \S\ref{section lambda 0 infty} tells us that the valency of the image of $\lambda_0(\infty)$ in $\overline{X}$ is 1. This fact hints that the quotient graph should be a ray. This is what we prove here below.\\

Note that $\mathrm{s}\in G(A)$ sends $r(\infty)$ to $r(0,0)$ (and $\lambda_0(\infty)$ to $\lambda_0(0,0)$), which suggests that the common image of these two rays gives the whole quotient graph $\overline{X}$. In order to prove this, we analyze the different steps in the proof of Theorem  \ref{principal result} for $(u,v)=(0,0)$.

We consider first Proposition \ref{prop stab}. We see that $\widetilde{\mathfrak{q}_{0,0}}=B$ and thus $\deg\widetilde{\mathfrak{q}_{0,0}}=0$. This implies that $N_0(0,0)=0$, whence the results of this proposition are valid for $n>0$. Moreover, after following the explicit computations, we get the description
\begin{equation*}
    \mathrm{Stab}(0,0,n)= \left\lbrace \mathrm{u}_{-a}(x,y)\widetilde{a}(t): (x,y) \in H(L,K)_{B}, \, \, t \in \mathbb{F}^{\times}, \, \,  \omega(y) \geq -\frac{n}{e_P} \right\rbrace, 
\end{equation*}
and therefore also
\begin{equation*}
    U(0,0,n)= \left\lbrace \mathrm{u}_{-a}(x,y): (x,y) \in H(L,K)_{B}, \, \, \omega(y) \geq -\frac{n}{e_P}  \right\rbrace.
\end{equation*}
In particular, this implies that $I(0,0)$ equals $B$, whence $\deg(I(0,0))=0$. 
Now, note that the preceding equality implies that $M(0,0,n):=\pi_1(U(0,0,n))$ is contained in
\[I(0,0)\left [\frac n2\right]=\left\lbrace x \in I(0,0): \omega(x) \geq -\frac{n}{2e_P} \right\rbrace.\]
We claim that, in this case, these sets actually coincide. Indeed, for each $x \in I(0,0)=B$ such that $\omega(x) \geq -\frac{n}{2e_P}$ we can define $y:=-N(x)/2$, so that $\mathrm{u}_{-a}(x,y)\in U(0,0,n) $. In other words, we have that $M(0,0,n)=I(0,0)[\frac n2]=B[\frac n2]$. In particular, following the proof of Lemma \ref{lemma ideals} we obtain $N_1(0,0)=N_0(0,0)=0$, whence
\[N(0,0)= \max\left\lbrace  \frac{2}{f_Pd}\Big(\deg\big(I(0,0)\big)+2g_D-1\Big), N_1(0,0)\right\rbrace=0.\]
Thus, we conclude the ray $\cf(0,0)$ starts in $\overline{\lambda_1(0,0)}$. Moreover, Lemma \ref{lemma valency cusps} tells us that the valency of $\overline{\lambda_1(0,0)}$ is at most two.

Finally, we look at Proposition \ref{prop cusps}. Since $\pic(B)=0$, we get that there is only one cusp and thus $\overline{X}$ is the union of $\cf(0,0)$ and a connected graph $Y$ containing $\overline{\lambda_0(0,0)}$. Now, consider $y\in Y\smallsetminus\cf(0,0)=Y\smallsetminus\{\overline{\lambda_1(0,0)}\}$. Then there is a shortest path in $Y$ leading to $\overline{\lambda_1(0,0)}$. However, since this last vertex has valency 2 and only one of its edges is connected to $Y$ via $\overline{\lambda_0(0,0)}$, we see that the path has to go through $\overline{\lambda_0(0,0)}$. But since this last vertex has valency 1 by the computations from \S\ref{section lambda 0 infty}, we see that $y$ must be $\overline{\lambda_0(0,0)}$. This proves that the graph $\overline{X}$ is a ray starting in $\overline{\lambda_0(0,0)}$.\\

Next theorem follows immediately from the previous discussion and Bass-Serre Theory (cf.~\cite[Chap. I, \S 5, Theo. 13]{S}). This is an analogue of Nagao's theorem in the context where $G=\su(h)$ (cf.~\cite[Chap. II, \S 1, Theo. 6]{S}).

\begin{Thm}\label{thm Nagao++}
The group $G(\mathbb{F}[t])$ is the sum of the groups $G(C)$ and 
$$\mathcal{B}(\mathbb{F}[t]):= \mathcal{U}_{a}(\mathbb{F}[t]) \rtimes \mathcal{T}(C)=\left\lbrace \mathrm{u}_{a}(x,y)\widetilde{a}(t): (x,y) \in H(L,K)_{B}, \, \, t \in \mathbb{F}^{\times} \right\rbrace,$$ amalgamated along their intersection.
\end{Thm}

\subsection{An example where \texorpdfstring{$L/K$}{L/K} is unramified at \texorpdfstring{$P$}{P}}

As in the last example, we let $C=\mathbb{P}_{\mathbb{F}}^1$ with $K=\mathbb{F}(C)=\mathbb{F}(t)$. Let $P$ be the point at infinity of $C$, and let
$\phi:\mathbb{P}_{\mathbb{F}}^1\rightarrow\mathbb{P}_{\mathbb{F}}^1$ be the automorphism satisfying $t\circ\phi=1/t$.
Since the extension $\mathbb{F}(\sqrt{a-t})/K$ is unramified at $P_0=\phi(P)$, and non-split when $a \in \mathbb{F} \smallsetminus \mathbb{F}^2$, we see that $L=\mathbb{F}(\sqrt{a-1/t})=K(\sqrt{at^2-t})$ is unramified and non-split at $P$. In particular $D\cong\mathbb{P}^1_\mathbb{F}$. Then $A=\mathbb{F}[t]$, $d=\deg(P)=1$ and $B=\mathbb{F}[t,\sqrt{at^2-t}]$. Moreover, in our context we have $K_P=\mathbb{F}((1/t))$, $L_Q=K_P(\sqrt{a-1/t})$, $\mathcal{O}_P=\mathbb{F}[[1/t]]$ and $\mathcal{O}_Q=\mathcal{O}_P[\sqrt{a-1/t}]$. In all that follows we fix the uniformizing parameter $\pi=1/t$ of $K_P$ and $L_Q$, and set $\kappa_Q=\mathbb{F}(\sqrt{a})\cong\mathcal{O}_Q/\pi\mathcal{O}_Q$. 

Note that $\mathrm{Pic}(B) = \mathbb{Z}/2\mathbb{Z}$, since $L$ is the function field of a plane quadratic curve, which is rational, and the point at infinity has degree $2$. In particular, the quotient graph $\overline{X}$ has only two cusps by Proposition \ref{prop cusps}. We claim that these two cusps correspond to $\cf(\infty)=\cf(0,0)$ and $\cf(0,v)$ for $v=\sqrt{a-1/t}\in L$. Indeed, note first that $\tr(v)=0$ and hence $(0,v)\in H(L,K)$. Then, using Proposition \ref{prop cusps} (and its proof), it suffices to see that the isotropic lines generated by $(1,0,0)$ and $(0,v,1)$ do not lie in the same $G(A)$-orbit, which amounts to show that the $B$-ideal
\[J=\{\lambda\in L\mid \lambda(0,v,1)\in B^3\subset L^3\}=B \cap \frac{1}{v} B,\]
is not principal. We have $t, \sqrt{at^2-t} \in B \cap \frac{1}{v}B$. So $M:=tB+\sqrt{at^2-t}B$ is contained in $B \cap \frac{1}{v}B$. Note that $B \cong \mathbb{F}[t,y]/(y^2-at^2+t)$, whence we get $B/M \cong \mathbb{F}[t,y]/(t,y,y^2-at^2+t) \cong \mathbb{F}[t,y]/(t,y) \cong \mathbb{F}$. This implies that $M$ is maximal. And since $1 \notin J$, we conclude that $J=M$.

Now, if $J=zB$ for some $z \in B$, then $z$ divides $t$ and $\sqrt{at^2-t}$. In particular, $N(z)$ divides $t^2$ and $at^2-t$, whence $N(z)$ divides $t$. If $N(z)$
has degree 1, we contradict the fact that the extension is unramified at infinity. 
We conclude that $N(z) \in \mathbb{F}^{*}$, in particular $z \in B$ is an invertible element and whence $J=zB=B$. This is a contradiction, which concludes the proof of our claim.\\

Having our explicit cusps at hand, we need to compute the numbers $N(0,0)$ and $N(0,v)$. As in the previous example, we get the identities $\widetilde{\mathfrak q_{0,0}}=I(0,0)=B$,
$$\mathrm{Stab}(0,0,n)= \left\lbrace \mathrm{u}_{-a}(x,y)\widetilde{a}(t): (x,y) \in H(L,K)_{B}, \, \, t \in \mathbb{F}^{\times}, \, \,  \omega(y) \geq -n \right\rbrace ,$$
and $M(0,0,n)=B[\frac n2]$. We conclude in the same way that $N(0,0)=0$. This implies that the tip of $\cf(0,0)$ is $\overline{\lambda_1(0,0)}$. On the other hand, we have already seen that $\widetilde{\mathfrak q_{0,v}}=B \cap \frac{1}{\bar v}B$ has degree 1, as it is a maximal ideal
with quotient $\bb F$. We obtain, then, that $N_0(0,v)$ is equal to 2. Following, thus, the explicit computations in the proof of Proposition \ref{prop stab}, we conclude, 
for any $n>2$, the identity
$$
\mathrm{Stab}(0,v,n)= \Big\{  g_{0,v}^{-1}\mathrm{u}_{a}(x,y) \widetilde{a}(t) g_{0,v} : (x,y) \in H(L,K)_{B}, \, t \in \mathbb{F}^{\times}, \,  \omega(y) \geqslant -n \Big\}.
$$
Therefore, the group containing all unipotent elements of $\mathrm{Stab}(0,v,n)$ is
$$ U(0,v,n)= \Big\{ g_{0,v}^{-1} \mathrm{u}_{a}(x,y) g_{0,v}: \, (x,y) \in H(L,K) \cap (B \times B), \,  \omega(y) \geqslant -n \Big\}.$$
In particular, we can fix $I(0,v)=B$. Moreover, as in the previous example we can show $M(0,v,n)=B[\frac n2]$. In particular, this implies that $N_1(0,v)=N_0(0,v)=2$. Thus,
\[N(0,v)= \max\left\lbrace  \frac{2}{f_Pd}(\deg(I(0,v))+2g_D-1), N_1(0,v)\right\rbrace=2,\]
and thus the tip of $\cf(0,v)$ is $\overline{\lambda_3(0,v)}$.\\

With all this, we are only left with the study of the ``body of the spider'', that is, the subgraph $Y$ of Theorem \ref{principal result}. This is however quite dependent on the base field $\bb F$ (already the valency of $\overline{\lambda_0(\infty)}$, computed in \S\ref{section lambda 0 infty}, illustrates this), so we will stop our analysis here. Note however that we already know that $Y$ must contain both,
$\overline{\lambda_0(\infty)}=\overline{\lambda_0(0,0)}$ and $\overline{\lambda_i(0,v)}$, for $i=0,1,2$, and that $\lambda_0(0,v)=\lambda_0(\infty)$. Indeed, it suffices to note that $g_{0,v}$, which satisfies $\lambda_0(0,v)=g_{0,v}^{-1} \cdot \lambda_0(\infty)$, is actually contained in the stabilizer of $\lambda_0(\infty)$.\\

We may summarize our findings by stating that $\overline{X}$ is obtained by attaching a graph $Y$ attached via three points ($\overline{\lambda_i(0,v)}$, for $i=0,1,2$) to a graph that is isomorphic to an appartment of the Bruhat-Tits tree $X$ (it is actually the image of the appartment $g_{0,v}^{-1}\bb A$, where $\bb A$ is the standard appartment).  Moreover, if $\bb F$ is finite, then the subgraph $Y$ is finite. On the other hand, when $\mathbb{F}=\mathbb{R}$, then $\overline{\lambda_0(0,0)}=\overline{\lambda_0(0,v)}$ need 
not be attached to $Y$, as its valency was proven to be 2 
in \S\ref{section lambda 0 infty}.
This is described by Figure \ref{Figure unramified case}.

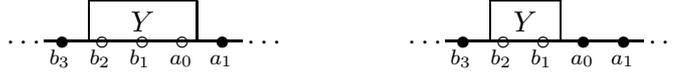
\begin{figure}[h!]
\linethickness{0.4pt}
\ifx\plotpoint\undefined\newsavebox{\plotpoint}\fi % GNUPLOT compatibility
\begin{picture}(46.5,12.25)(-100,0)
\put(21,5){\framebox(40,15)[cc]{$Y$}}
\put(4,5){\line(1,0){74}}
\put(8,2){$\bullet$}
\put(6,-2){${}_{b_3}$}
\put(23,2){$\circ$}
\put(21,-2){${}_{b_2}$}
\put(38,2){$\circ$}
\put(36,-2){${}_{b_1}$}
\put(53,2){$\circ$}
\put(51,-2){${}_{a_0}$}
\put(68,2){$\bullet$}
\put(66,-2){${}_{a_1}$}
\put(-10,2){$\cdots$}
\put(80,2){$\cdots$}
\end{picture}
\begin{picture}(46.5,12.25)(-200,0)
\put(21,5){\framebox(26,15)[cc]{$Y$}}
\put(4,5){\line(1,0){74}}
\put(8,2){$\bullet$}
\put(6,-2){${}_{b_3}$}
\put(23,2){$\circ$}
\put(21,-2){${}_{b_2}$}
\put(38,2){$\circ$}
\put(36,-2){${}_{b_1}$}
\put(53,2){$\bullet$}
\put(51,-2){${}_{a_0}$}
\put(68,2){$\bullet$}
\put(66,-2){${}_{a_1}$}
\put(-10,2){$\cdots$}
\put(80,2){$\cdots$}
\end{picture}
\caption{The quotient graph $\overline{X}$ of Example 12.2 for arbitrary $\bb F$ (left), and for $\bb F=\bb R$ (right). Here $a_i=\lambda_i(0,0)$ and $b_i=\lambda_i(0,v)$.}\label{Figure unramified case}
\end{figure}

The following result is then immediate from Bass-Serre Theory (cf.~\cite[Chap. I, \S 5, Theo. 13]{S}) and our analysis of stabilizers of vertices in \S\ref{sec stab}.

\begin{Thm}
Assume that $\mathbb{F}$ is finite. Set $C=\mathbb{P}_{\mathbb{F}}^1$, $K=\mathbb{F}(C)=\mathbb{F}(t)$, $L=\mathbb{F}(\sqrt{a-1/t})$ and let $P$ be the 
point at infinity in $C$. Let $G(\mathbb{F}[t])$ be the abstract group consisting of the $\mathbb{F}[t]$-points of the special unitary group $G$ defined from the quadratic extension $L/K$ and the form $h_K$. Then, there exist: 
\begin{itemize}
    \item a finitely generated subgroup $\cal H$ of $G(\mathbb{F}[t])$; and
    \item two finite subgroups $\mathcal{B}_1, \mathcal{B}_2$, defined as the respective intersection of $\cal H$ with $$ \mathcal{P}_1= \mathcal{U}_{a}(\mathbb{F}[t]) \rtimes \mathcal{T}(C)=\Big\{ \mathrm{u}_{a}(x,y) \widetilde{a}(t) : (x,y) \in H(L,K)_{B}, \, t \in \mathbb{F}^{\times} \Big\}, $$
    and 
    $$\mathcal{P}_2= g_{0,v}^{-1}(\mathcal{U}_{a}(\mathbb{F}[t]) \rtimes \mathcal{T}(C)) g_{0,v}=\Big\{  g_{0,v}^{-1}\mathrm{u}_{a}(x,y) \widetilde{a}(t) g_{0,v} : (x,y) \in H(L,K)_{B}, \, t \in \mathbb{F}^{\times} \Big\},$$
\end{itemize}
such that $G(\mathbb{F}[t])$ is the sum of $H$, $\mathcal{P}_1$ and $\mathcal{P}_2$, amalgamated along $\mathcal{B}_1$ and $\mathcal{B}_2$.
\end{Thm}

\section*{Acknowledgements}
The authors would like to thank an anonymous referee for his or her thorough reading of this paper and his or her comments.

The first author's research was partially supported by ANID via FONDECYT Grant N\textsuperscript{o} 1200874. The second author's research was partially supported by CONICYT/ANID via the Doctoral fellowship N\textsuperscript{o} 21180544.
The third author was partially supported by the GeoLie project (ANR-15-CE40-0012, The French National Research Agency).
The fourth author's research was partially supported by ANID via FONDECYT Grant N\textsuperscript{o} 1210010.

\end{document}